\documentclass{imanum}
\usepackage{graphicx}

\usepackage{array}
\usepackage{amsmath}
\usepackage{multirow}
\usepackage{booktabs}

\usepackage[caption=false]{subfig}
\usepackage{pgfplots}

\usepackage{algorithm} 
\usepackage{algorithmic} 

\usepackage{graphicx,epstopdf}

\usepackage{lineno}
\usepackage{color}
\usepackage{url}

\jno{drnxxx}
\received{}
\revised{}

\begin{document}

\title{Convergence Rate of Inertial Forward-Backward Splitting Algorithms Based on the Local Error Bound Condition}
\shorttitle{Convergence Rate of IFBS}

\author{%
	{\sc
		Hongwei Liu\thanks{Email: hwliuxidian@163.com},
		Ting Wang\thanks{Corresponding author. Email: wangting\_7640@163.com}
	} \\[2pt]
	School of Mathematics and Statistics, Xidian University, Xi'an, 710126,  China\\[6pt]
	{\sc and}\\[6pt]
	{\sc Zexian Liu}\thanks{Email: liuzexian2008@163.com}\\[2pt]
	State Key Laboratory of Scientific and Engineering Computing, Institute of Computational Mathematics and Scientific/Engineering computing, AMSS, Chinese Academy of Sciences, Beijing, 100190, China;\\
	School of Mathematics and Statistics, Guizhou  University, Guiyang, 550025, China
}
\shortauthorlist{Convergence Rate of IFBS}

\maketitle

\begin{abstract}	{The `` Inertial Forward-Backward algorithm '' (IFB) is a powerful tool for convex nonsmooth minimization problems, and under the local error bound condition, the $R$-linear convergence rates for the sequences of objective values and iterates have been proved if the inertial parameter $\gamma_k$ satisfies ${\sup _k}{\gamma _k} < 1.$ However, the convergence result for ${\sup _k}{\gamma _k} = 1$ is not know. In this paper, based on the local error bound condition, we exploit a new assumption condition for the important parameter $t_k$ in IFB, which implies that ${\lim _{k \to \infty }}{\gamma _k} = 1,$ and establish the convergence rate of function values and strong convergence of the iterates generated by the IFB algorithms with six $t_k$ satisfying the above assumption condition in Hilbert space. It is remarkable that, under the local error bound condition, we show that the IFB algorithms with some $t_k$ can achieve sublinear convergence rate of $o\left( {\frac{1}{{{k^p}}}} \right)$ for any positive integer $p>1$. In addition, we propose a class of Inertial Forward-Backward algorithm with adaptive modification and show it has same convergence results as IFB under the error bound condition. Some numerical experiments are conducted to illustrate our results.
}
{Inertial Forward-Backward algorithm; local error bound condition; rate of convergence.}
\end{abstract}

\section{Introduction}
\label{sec;introduction}

Let $H$ be a real Hilbert space.  $f:H \to \mathbb{R}$ be a smooth convex function and continuously differentiable with $L_f$-Lipschitz continuous gradient, and $g:H \to \mathbb{R} \cup \left\{ { + \infty } \right\}$ be a proper lower semi-continuous convex function. We also assume that the proximal operator of $\lambda g,$ i.e., 
\begin{equation}
{\rm{pro}}{{\rm{x}}_{\lambda g}}\left(  \cdot  \right)\mathop { = \arg \min }\limits_{x \in H} \left\{ {g\left( x \right) + \frac{1}{{2\lambda }}{{\left\| {x -  \cdot } \right\|}^2}} \right\} 
\end{equation}
can be easliy computed for all $\lambda >0.$
In this paper, we consider the following problem:
\begin{displaymath}
(P) \quad \quad \quad \quad \quad \mathop {\min }\limits_{x \in H} F\left( x \right) := f\left( x \right) + g\left( x \right).
\end{displaymath}
We assume that problem ($P$) is solvable, i.e., ${X}: = \arg \min F \ne \emptyset ,$ and for ${x_* } \in {X}$ we set ${F_* }: = F\left( {{x_* }} \right).$

In order to solve the problem ($P$), several algorithms have been proposed based on
the use of the proximal operator due to the non differentiable part. One can consult \citet{johnstone2017local}, \citet{IFB} and \citet{villa2013accelerated} for a recent account on the proximal-based algorithms that
play a central role in nonsmooth optimization. A typical optimization strategy for solving problem ($P$) is the Inertial Forward-Backward algorithm (IFB), which consists in applying iteratively at every point the non-expansive operator ${T_\lambda }:H \to H,$ defined as 
\[{T_\lambda }\left( x \right) := {\rm{pro}}{{\rm{x}}_{\lambda g}}\left( {x - \lambda \nabla f\left( x \right)} \right)\;\forall x \in H.\]
\begin{algorithm}[tbhp]
	\caption{ Inertial Forward-Backward algorithm (IFB) }		
	\hspace*{0.2cm} \textbf{Step 0.} Take ${y_1} = {x_0} \in {R^n},{t_1} = 1.$ Input $\lambda  = \frac{\mu }{{{L_f}}},$ where ${\mu } \in \left] {{\rm{0}},{\rm{1}}} \right[$.\\
	\hspace*{0.7cm} \textbf{Step k.}  Compute \\
	\hspace*{0.2cm} \hspace*{2cm} ${x_k} = {T_\lambda }\left( {{y_k}} \right) = {\rm{pro}}{{\rm{x}}_{\lambda g}}\left( {{y_k} - \lambda \nabla f\left( {{y_k}} \right)} \right)$\\
	\hspace*{0.2cm} \hspace*{2cm} ${y_{k + 1}} = {x_k} + {\gamma _k}\left( {{x_k} - {x_{k - 1}}} \right)$ where ${\gamma _k} = \frac{{{t_k} - 1}}{{{t_{k + 1}}}}.$
\end{algorithm}
In view of the composition of IFB, we can easily \textcolor{red}{find} that the inertial term ${\gamma _k}$ plays an important role for improving the speed of convergence of IFB.
Based on Nesterov's extrapolation \textcolor{red}{technique} \citep[see,][]{Nesterov}, Beck and Teboulle proposed a ``fast iterative shrinkage-thresholding algorithm'' (FISTA) with $t_1 =1$ and ${t_{k + 1}}{\rm{ = }}\frac{{{\rm{1 + }}\sqrt {{\rm{1 + 4}}t_k^2} }}{2}$ for solving ($P$) \citep[see,][]{FISTA}. The remarkable \textcolor{red}{properties} of this algorithm \textcolor{red}{are} the computational simplicity and the significantly better global rate of convergence of the function values, that is $F\left( {{x_k}} \right) - F\left( {{x_* }} \right) = O\left( {\frac{1}{{{k^2}}}} \right).$ Several variants of FISTA considered in works such as \citet{apidopoulos2020}, \citet{Convea}, \citet{Stepsize2019}, \citet{anotherLook}, \citet{appl}, \citet{O(1/k3)} and \citet{localCov}, the properties such as convergence of the iterates and rate of convergence of function values have also been studied. 

Chambolle and Dossal \citep[see,][]{Convea} pointed out that FISTA satisfies a better worst-case estimate, however,  the convergence of the iterates is not known. They proposed a new ${t_k} = \frac{{k - 1 + a}}{a}\left( {a > 2} \right)$ to show that the iterates generated by the corresponding IFB, named ``FISTA\_CD", converges weakly to the minimizer of $F$. 
Attouch and Peypouquet \citep[see,][]{cov_CD} further proved that the sequence of function values generated by FISTA\_CD \textcolor{red}{approximates} the optimal value of the problem with a rate that is strictly faster than $O\left( {\frac{1}{{{k^2}}}} \right),$ namely $F\left( {{x_k}} \right) - F\left( {{x_* }} \right) = o\left( {\frac{1}{{{k^2}}}} \right).$ 
Apidopoulos et al. \citep[see,][]{apidopoulos2020} noticed that the basic idea of the choices of $t_k$ in \citet{Attouch2018}, \citet{FISTA} and \citet{Convea} is the Nesterov's rule: $t_k^2 - t_{k + 1}^2 + {t_{k + 1}} \ge 0,$
and they \textcolor{red}{focused} on the case that the Nesterov's rule is not satisfied. They studied the ${\gamma _k} = \frac{n}{{n + b}}$ with $ {0 < b < 3}$ and found that the exact estimate bound is:
$F\left( {{x_k}} \right) - F\left( {{x_* }} \right) = O\left( {\frac{1}{{{k^{\frac{{2b}}{3}}}}}} \right)$. Attouch and Peypouquet \citep[see,][]{Attouch2018} considered various options of ${\gamma _k}$ to analyze the convergence rate of the function values and weak convergence of the iterates under the given assumptions. Further, they showed that the strong convergence of iterates can be satisfied for the special options of $f$.
Wen, Chen and Pong \citep[see,][]{wenbo} showed that for the nonsmooth convex minimization problem ($P$),  under the local error bound condition \citep[see,][]{Tseng}, the $R$-linear convergence of both the sequence $\left\{ {{x_k}} \right\}$ and the corresponding sequence of objective values $\left\{ {F\left( {{x_k}} \right)} \right\}$ can be satisfied if ${\sup _k}{\gamma _k} < 1;$ and they pointed out that the sequences $\left\{ {{x_k}} \right\}$ and $\left\{ {F\left( {{x_k}} \right)} \right\}$ generated by FISTA with fixed restart or both fixed and adaptive restart schemes \citep[see,][]{Restart} are $R$-linearly convergent under the error bound condition. However, the local convergence rate of the iterates generated by FISTA for solving ($P$) is still unknown, even under the local error bound condition.

The local error bound condition, which estimates the distance from $x$ to $X^*$ by the norm of the proximal residual at $x,$ \textcolor{red}{has} been proved to be extremely useful in analyzing the convergence rates of a host of iterative methods for solving optimization problems \citep[see,][]{zhou2017a}. Major contributions on developing and using error bound condition to derive convergence  results of iterative algorithms have been developed in a series of papers \citep[see, e.g.][]{hai2020error,Luo1992,necoara2019linear,Tseng,tseng2010,tseng2010approximation,zhou2017a}. Zhou and So \citep[see,][]{zhou2017a} established error bounds for minimizing the sum of a smooth convex function and a general closed proper convex function. Such a problem contains general constrained minimization problems and various regularized loss minimization formulations in machine learning, signal processing, and statistics.
There are many choices of $f$ and $g$ \textcolor{red}{satisfy} the local error bound condition, including: 
\begin{itemize}
	\item \textcolor{red}{\citep[Theorem 3.1]{Pang1987}} $f$ is strong convex, and $g$ is arbitrary.
	\item \textcolor{red}{\citep[Theorem 2.3]{Luo1992a}} $f$ is a quadratic function, and $g$ is a polyhedral function.
	\item \textcolor{red}{\citep[Theorem 2.1]{Luo1992}} $g$ is a polyhedral function and $f = h(Ax){\rm{ + }}\left\langle {c,x} \right\rangle ,$ where $A \in {R^{m \times n}},c \in {R^n}$ and $h$ is a continuous differentiable function with gradient Lipschitz continuous and strongly convex on any compact convex set. This covers the well-known LASSO.
	\item \textcolor{red}{\citep[Theorem 4.1]{Luo1993}} $g$ is a polyhedral function and $f\left( x \right) = \mathop {\max }\limits_{y \in Y} \left\{ {{{\left( {Ax} \right)}^T}y - h\left( y \right)} \right\} + {q^T}x,$ where $Y$ is a polyhedral set, $h$ is a strongly convex differentiable function with gradient Lipschitz continuous.
	\item \textcolor{red}{\citep[Theorem 2]{tseng2010approximation}} $f$ takes the form $f\left( x \right) = h\left( {Ax} \right),$ where $h$ is same as the above second item and $g$ is the grouped LASSO regularizer.
\end{itemize}
 More examples satisfying the error bound condition can be referred to \citet{tseng2010approximation}, \citet{zhou2017a}, \citet{Tseng}, \citet{Pang1987} and \citet{Luo1992}.

It has been observed numerically that first-order methods for solving those specific structured instances of problem ($P$) converge at a much faster rate than that suggested by the theory in \citet{localCov}, \citet{xiao2013a} and \citet{zhou2017a}. A very powerful approach to analyze this phenomenon is the local error bound condition. Hence, the first point this work focuses is the improved convergence rate of IFBs with some special $t_k$ under the local error bound condition.

We also pay attention to the Nesterov's rule: $t_k^2 - t_{k + 1}^2 + {t_{k + 1}} \ge 0.$ For the $t_k$ \textcolor{red}{satisfies} it, we can derive that ${t_{k + 1}} - {t_k} < 1$ and $\sum\limits_{k = 1}^{ + \infty } {\frac{1}{{{t_k}}}}$ is divergent, which will greatly limit the choice of $t_k.$ What \textcolor{red}{we} expect is whether we can find the more suitable $t_k$ and obtain the improved theoretical results if we replace the Nesterov's rule by some new we proposed.

\textbf{Contributions.}

In this paper, based on the local error bound condition, we exploit an assumption condition for the important parameter $t_k$ in IFB, and prove the convergence results including convergence rate of function values and strong convergence of iterates generated by the corresponding IFB. The above mentioned assumption condition imposed on $t_k$ \textcolor{red}{provides} a theoretical basis for choosing a new $t_k$ in IFB to solving those problems satisfying the local error bound condition, like LASSO.  We use a ``comparison methods'' to discuss six choices of $t_k$, which \textcolor{red}{include} the ones in original FISTA \citep[see,][]{FISTA} and FISTA\_CD \citep[see,][]{Convea} and \textcolor{red}{satisfy} our assumption condition, and separately show the improved convergence rates of the function values and establish the sublinear convergence of the iterates generated by corresponding IFBs. We also establish the same convergence results for IFB with an adaptive modification (IFB\_AdapM), which performs well in numerical experiments. It is remarkable that, under the local error bound condition, the strong convergence of the iterates generated by the original FISTA is established, the convergence rate of function values for FISTA\_CD is improved to $o\left( {\frac{1}{{{k^{2(a+1)}}}}} \right)$, and the IFB algorithms with some $t_k$ can achieve sublinear convergence rate  $o\left( {\frac{1}{{{k^p}}}} \right)$ for any positive integer $p>1$.

\section{ An new assumption condition for $t_k$ and the convergence of the corresponding IFB algorithms  }

In this section, we derive a new assumption condition for the $t_k$ in IFB, and analyze the convergence results of the corresponding IFB under the local error bound condition.

We start by recalling a key result, which plays an important role in our theoretical analysis. 
\begin{lemma}\label{l4}
	\textcolor{red}{\citep[ineq (4.36)]{Chambolle2016Pock}} For any $y \in {R^n},\lambda  = \frac{\mu }{{{L_f}}},$ where ${\mu } \in \left( {{\rm{0}},{\rm{1}}} \right]$, we have,
	\begin{equation}\label{O1}
	\forall x \in \mathbb{R}^n \quad F\left( {{T_\lambda }\left( y \right)} \right) \le F\left( x \right) + \frac{1}{{2\lambda }}{\left\| {x - y} \right\|^2} - \frac{{1 - \mu }}{{2\lambda }}{\left\| {{T_\lambda }\left( y \right) - y} \right\|^2} - \frac{1}{{2\lambda }}{\left\| {{T_\lambda }\left( y \right) - x} \right\|^2}. 
	\end{equation}
	
\end{lemma}	

Next, we give a very weak assumption to show that the sequence $\left\{ {F\left( {{x_k}} \right)} \right\},$ which is generated by Algorithm 1 with $0 \le \gamma _k \le 1$ for $k$ is large sufficiently, converges to $F\left( {{x_* }} \right)$ \textcolor{red}{independent} on $t_k.$

\textbf{Assumption $A_0:$} For any $\xi_0  \ge {F^ * },$ there exist $\epsilon_0 > 0$ and ${ \tau_0 } > 0$ such that 
\begin{equation}
{\rm{dist}}\left( {x,{X^*}} \right) \le {\tau _{\rm{0}}}
\end{equation}
whenever $\left\| {{T_{\frac{1}{{{L_f}}}}}\left( x \right) - x} \right\| < {\varepsilon _0}$ and $F\left( x \right) \le \xi_0 .$

\emph{Remark 2.} Note that Assumption $A_0$ can be derived by the assumption that $F$ is boundedness of level sets. 
\begin{lemma}\label{L3.5}
	\textcolor{red}{\citep[Lemma 2]{Nesterov2013}} For ${\lambda_1} \ge {\lambda_2} > 0,$ we have 
	\begin{equation}\label{F49}
	\forall x \in \mathbb{R}^n \quad \left\| {{T_{{\lambda _1}}}\left( x \right) - x} \right\| \ge \left\| {{T_{{\lambda _2}}}\left( x \right) - x} \right\|\quad {\rm{and}}\quad \frac{{\left\| {{T_{{\lambda _1}}}\left( x \right) - x} \right\|}}{{{\lambda _1}}} \le \frac{{\left\| {{T_{{\lambda _2}}}\left( x \right) - x} \right\|}}{{{\lambda _2}}}.
	\end{equation}
\end{lemma}
\begin{proof}
	\textcolor{red}{Above lemma can be obtained from Lemma 2 of \citet{Nesterov2013} with $B:=I,$ $L: = \frac{1}{\lambda}.$ }
\end{proof}
\begin{theorem}\label{T2.0}
	Let $\left\{ {{x_k}} \right\},$ $\left\{ {{y_k}} \right\}$ be generated by Algorithm 1. Suppose that Assumption $A_0$ holds and there exists a positive interger $n_0$ such that for $k \ge n_0,$ $0 \le \gamma _k \le 1.$ Then, \\
	1) $\sum\limits_{k = 1}^\infty  {{{\left\| {{x_{k + 1}} - {y_{k + 1}}} \right\|}^2}}$ is convergent.\\
	2) $\mathop {\lim }\limits_{k \to \infty } F\left( {{x_k}} \right) = F\left( {{x^ * }} \right).$
\end{theorem}
\begin{proof} 
	Applying the inequality (\ref{O1}) at the point $x={x_k},\;y={y_{k + 1}},$ we obtain 
	\begin{equation}\label{O45}
	\forall k \ge 1 \quad \frac{{1 - \mu }}{{2\lambda }}{\left\| {{x_{k + 1}} - {y_{k + 1}}} \right\|^2} \le \left( {F\left( {{x_k}} \right) + \frac{{\gamma _k^2}}{{2\lambda }}{{\left\| {{x_k} - {x_{k - 1}}} \right\|}^2}} \right) - \left( {F\left( {{x_{k + 1}}} \right) + \frac{1}{{2\lambda }}{{\left\| {{x_{k + 1}} - {x_k}} \right\|}^2}} \right).
	\end{equation}
	Then, we can easily obtain $\sum\nolimits_{k = {n_0}}^\infty  {{{\left\| {{x_{k + 1}} - {y_{k + 1}}} \right\|}^2}}  <  + \infty $ since that $0 \le \gamma _k \le 1$ holds for any $k \ge n_0.$ Then, result 1) can be obtained since that increasing the finite term does not change the convergence of the series. Moreover, for any $\bar \epsilon > 0,$ there exists a $n_1,$ which is sufficiently large, such that for any $k \ge \bar n:=\max \left( {{n_0},{n_1}} \right),$ $\left\| {{x_k} - {y_k}} \right\| < \bar \epsilon .$
	Setting ${\xi _0} = F\left( {{x_{{\bar n} + 1}}} \right) + \frac{1}{{2\lambda }}{\left\| {{x_{{\bar n} + 1}} - {x_{{\bar n}}}} \right\|^2}.$
	From Lemma \ref{L3.5} with $\lambda  < \frac{1}{{{L_f}}}$ and the nonexpansiveness property of the proximal operator, we obtain that 
	\begin{equation}\label{O47}
	\left\| {{T_{\frac{1}{{{L_f}}}}}\left( {{x_k}} \right) - {x_k}} \right\| \le \frac{1}{{\lambda {L_f}}}\left\| {{T_{\lambda}}\left( {{x_k}} \right) - {x_k}} \right\|
	= \frac{1}{{\lambda {L_f}}}\left\| {{T_{\lambda}}\left( {{x_k}} \right) - {T_{\lambda}}\left( {{y_k}} \right)} \right\| \le \left( {1 + \frac{1}{{\lambda {L_f}}}} \right)\left\| {{x_k} - {y_k}} \right\|,
	\end{equation}
	hence, $\left\| {{T_{\frac{1}{{{L_f}}}}}\left( {{x_k}} \right) - {x_k}} \right\| < \left( {1 + \frac{1}{{\lambda {L_f}}}} \right)\bar \epsilon$ for any $k \ge \bar n.$ Also,
	it follows from (\ref{O45}) that for any $k \ge \bar n,$ $\left\{ {F\left( {{x_{k + 1}}} \right) + \frac{1}{{2\lambda }}{{\left\| {{x_{k + 1}} - {x_k}} \right\|}^2}} \right\}$ is non-increasing, then, $F\left( {{x_k}} \right) \le {\xi _0}.$ Hence, combining with the Assumption $A_0$, we have for ${\xi _0} = F\left( {{x_{{\bar n} + 1}}} \right) + \frac{1}{{2\lambda }}{\left\| {{x_{{\bar n} + 1}} - {x_{{\bar n}}}} \right\|^2},$ there exist ${\epsilon _0}: = \left( {1 + \frac{1}{{\lambda {L_f}}}} \right)\bar \epsilon $ and ${ \tau_0 } > 0,$ such that 
	\begin{equation}\label{O46}
	\forall k \ge \bar n, \quad {\rm{dist}}\left( {{x_k},{X^ * }} \right) \le \tau_0.
	\end{equation}
	In addition, applying the inequality (\ref{O1}) at the point $y={y_{k + 1}},$ and $x$ be an $x_{k+1}^* \in {X }$ such that ${\rm{dist}}\left( {{x_{k+1}},{X }} \right) = \left\| {{x_{k+1}} - x_{k+1}^*} \right\|,$ we obtain
	\begin{align}
	&F\left( {{x_{k + 1}}} \right) - F\left( {{x_*}} \right) \label{O48} \\
	&\le \frac{1}{{2\lambda }}{\left\| {{y_{k + 1}} - x_{k + 1}^*} \right\|^2} - \frac{1}{{2\lambda }}{\left\| {{x_{k + 1}} - x_{k + 1}^*} \right\|^2} \nonumber\\
	&= \frac{1}{{2\lambda }}{\left\| {{y_{k + 1}} - {x_{k + 1}}} \right\|^2} + \frac{1}{\lambda }\left\langle {{y_{k + 1}} - {x_{k + 1}},{x_{k + 1}} - x_{k + 1}^*} \right\rangle  \nonumber\\
	&\le \frac{1}{{2\lambda }}{\left\| {{y_{k + 1}} - {x_{k + 1}}} \right\|^2} + \frac{1}{\lambda }\left\| {{y_{k + 1}} - {x_{k + 1}}} \right\|{\rm{dist}}\left( {{x_{k + 1}},{X}} \right).\nonumber 
	\end{align} 
	Then, combining with $\left\| {{y_{k + 1}} - {x_{k + 1}}} \right\| \to 0$ by result 1) and (\ref{O46}), we have $\mathop {\lim }\limits_{k \to \infty } F\left( {{x_k}} \right) = F\left( {{x_* }} \right).$ 
\end{proof}

The rest of this paper is based on the following assumption.

\textbf{Assumption $A_1:$}  (``Local error bound condition", \citet{Tseng}) For any $\xi  \ge {F_* },$ there exist $\varepsilon  > 0$ and ${\bar \tau } > 0$ such that 
\begin{equation}
{\rm{dist}}\left( {x,{X^*}} \right) \le \bar \tau \left\| {{T_{\frac{1}{{{L_f}}}}}\left( x \right) - x} \right\|
\end{equation}
whenever $\left\| {{T_{\frac{1}{{{L_f}}}}}\left( x \right) - x} \right\| < \varepsilon $ and $F\left( x \right) \le \xi .$

As mentioned in Section 1, the $t_k$ in FISTA accelerates convergence rate from $O\left( {\frac{1}{k}} \right)$ to $O\left( {\frac{1}{{{k^2}}}} \right)$ for the function values and $t_k$ in FISTA\_CD improves the convergence rate to $o\left( {\frac{1}{{{k^2}}}} \right).$ Other options for $t_k$ are considered in \citet{Attouch2018} and \citet{apidopoulos2020}. Hence, we see that $t_k$ is the crucial factor to guarantee the convergence of the iterates or to improve the rate of convergence for the function values. Apidopoulos et al. in \textcolor{red}{\citet{apidopoulos2020}} points that if $t_k$ satisfies the Nesterov's rule, then one can obtain a better convergence rate. However, we notice that the Nesterov's rule will \textcolor{red}{limit} the choice of $t_k$ greatly. In the following, we present a new Assumption $A_2$ for $t_k,$ which \textcolor{red}{helps} us to obtain some new options of $t_k,$ and analyze the convergence of iterates and convergence rate of the function values for the Algorithm 1 with a class of abstract $t_k$ satisfied Assumption $A_2$ under the local error bound condition. 

\textbf{Assumption $A_2:$} There exists a positive constant $0 < \sigma  \le 1$ such that $\mathop {\lim }\limits_{k \to \infty } k^\sigma\left( {\frac{{{t_{k + 1}}}}{{{t_k}}} - 1} \right) = c,$ where $c>0.$

\emph{Remark 3.} It follows that ${\gamma _k} \in \left] {0,1} \right[,$ for any $k$ is sufficiently large, ${\mathop {\lim }\limits_{k \to \infty } {t_k} =  + \infty }$ and ${\mathop {\lim }\limits_{k \to \infty } \frac{{{t_{k + 1}}}}{{{t_k}}} = 1}$ from Assumptions $A_2$. (It is easy to verify that $t_k$ in FISTA and $t_k$ in FISTA\_CD both satisfy the Assumption $A_2$ by choosing $\sigma = 1$, also, we can see that there exist some $t_k$, which satisfy or do not satisfy Nesterov's rule, satisfy Assumption $A_2$ (See Section 3))

\begin{lemma}\label{l3}
	Suppose that Assumptions ${A_1}$ and $A_2$ hold. Let $\left\{ {{x_k}} \right\}$ be generated by Algorithm 1 and ${x_*} \in {X }.$ Then, there exists a constant $\tau _1 >0$ such that \[ \forall k \ge 1,\quad F\left( {{x_{k + 1}}} \right) - F\left( {{x_*}} \right) \le \frac{{{\tau _1}}}{\lambda }{\left\| {{y_{k + 1}} - {x_{k + 1}}} \right\|^2}.\]
\end{lemma}
\begin{proof} 
	Since ${\gamma _k} \in \left] {0,1} \right[$ holds for $\forall k \ge n_0$ by Assumption $A_2,$ then, similar with the proof of Theorem \ref{T2.0}, for ${\xi _0} = F\left( {{x_{{\bar n} + 1}}} \right) + \frac{1}{{2\lambda }}{\left\| {{x_{{\bar n} + 1}} - {x_{{\bar n}}}} \right\|^2},$ there exist ${\epsilon _0}: = \left( {1 + \frac{1}{{\lambda {L_f}}}} \right)\bar \epsilon $ and ${ \tau_0 } > 0,$ such that
	\begin{equation}\label{O12}
	{\rm{dist}}\left( {{x_k},{X^*}} \right) \le \bar \tau \left\| {{T_{\frac{1}{{{L_f}}}}}\left( {{x_k}} \right) - {x_k}} \right\| \le \bar \tau \left( {1 + \frac{1}{{\lambda {L_f}}}} \right)\left\| {{x_k} - {y_k}} \right\| \le \frac{{2\bar \tau }}{\mu }\left\| {{x_k} - {y_k}} \right\|,
	\end{equation}
	where the second inequality of (\ref{O12}) follows from (\ref{O47}) and the third one follows the fact that $\lambda  = \frac{\mu }{{{L_f}}}$ with $\mu  \in \left( {0,1} \right).$ In addition, it follows from (2.7) that
	\begin{multline}\label{O34}
	F\left( {{x_{k + 1}}} \right) - F\left( {{x_*}} \right) \le \frac{1}{{2\lambda }}{{\left\| {{y_{k + 1}} - {x_{k + 1}}} \right\|}^2} + \frac{1}{\lambda }\left\| {{y_{k + 1}} - {x_{k + 1}}} \right\|{\rm{dist}}\left( {{x_{k + 1}},{X}} \right) \\
	\le \frac{1}{\lambda }\left( {\frac{{2\bar \tau }}{\mu } + \frac{1}{2}} \right){{\left\| {{y_{k + 1}} - {x_{k + 1}}} \right\|}^2}, \quad \forall k \ge \bar n.
	\end{multline}
	\textcolor{red}{Also, we can find a constant $c>0$ such that for $\forall 1\le k \le {\bar n} -1,$ $F\left( {{x_{k + 1}}} \right) - F\left( {{x_*}} \right) \le \frac{c}{\lambda }{\left\| {{y_{k + 1}} - {x_{k + 1}}} \right\|^2}.$
	Therefore, there exists a ${\tau _1} \ge \max \left\{ {\frac{{2\bar \tau }}{\mu } + \frac{1}{2},c} \right\}$ such that the conclusion holds.} 
\end{proof}

Here, we introduce a new way, which we called ``comparison method'', that \textcolor{red}{considers} a sequence $\left\{ {{\alpha _k}} \right\}$ such that ${\alpha _k} = \frac{{{s_k} - 1}}{{{s_{k + 1}}}} \ge {\gamma _k},$ where $\left\{ {{s_k}} \right\}$ is a nonnegative sequence, to estimate the bounds of objective function and the local variation of the iterates.

\begin{lemma}\label{L2}
	Suppose that there exists a nonnegative sequence $\left\{ {{s_k}} \right\}$ such that ${\alpha _k} = \frac{{{s_k} - 1}}{{{s_{k + 1}}}} \ge {\gamma _k},$ for $k$ is sufficiently large, and ${\gamma _k} = \frac{{{t_k} - 1}}{{{t_{k + 1}}}},$ $t_k$ satisfies the Assumption $A_2.$ Then, we have $\mathop {\lim }\limits_{k \to \infty } {s_k} =  + \infty ,$ and $\mathop {\lim \sup }\limits_{k \to \infty } \frac{{s_{k + 1}^2 - s_k^2}}{{s_k^2}} \le 0.$
\end{lemma}
\begin{proof}
	See the detailed proof in Appendix \ref{A}.
\end{proof}
\begin{theorem}\label{L1}
	Suppose that Assumptions ${A_1}$ and ${A_2}$ hold and there exists a nonnegative sequence $\left\{ {{s_k}} \right\}$ such that  ${\alpha _k} = \frac{{{s_k} - 1}}{{{s_{k + 1}}}} \ge \gamma _k,$ for $k$ is sufficiently large. Then, we have that $F\left( {{x_{k + 1}}} \right) - F\left( {{x_* }} \right) = o\left( {\frac{1}{{s_{k + 1}^2}}} \right)$ and ${\left\| {{x_{k + 1}} - {x_k}} \right\|} = O\left( {\frac{1}{{s_{k + 1}}}} \right).$ Further, if $\sum\limits_{k = 1}^\infty  {\frac{1}{{s_{k+1}}}} $ is convergent, then the iterates $\left\{ {{x_k}} \right\}$ converges strongly to a minimizer of $F.$	
\end{theorem}
\begin{proof}
	Denote that ${E_k} = s_{k + 1}^2\left( {F\left( {{x_k}} \right) - F\left( {{x_* }} \right)} \right) + \frac{{s_k^2}}{{2\lambda }}{\left\| {{x_k} - {x_{k - 1}}} \right\|^2}.$ Applying (\ref{O45}), we have 
	\begin{equation*}
	F\left( {{x_{k + 1}}} \right) - F\left( {{x_* }} \right) + \frac{{1 - \mu }}{{2\lambda }}{\left\| {{x_{k + 1}} - {y_{k + 1}}} \right\|^2} + \frac{1}{{2\lambda }}{\left\| {{x_{k + 1}} - {x_k}} \right\|^2} \le F\left( {{x_k}} \right) - F\left( {{x_* }} \right) + \frac{{\gamma _k^2}}{{2\lambda }}{\left\| {{x_k} - {x_{k - 1}}} \right\|^2}.
	\end{equation*}
	By the assumption condition, we have $\gamma _k^2 \le \alpha _k^2$ for any $k$ is sufficiently large, then, 
	\begin{equation*}
	F\left( {{x_{k + 1}}} \right) - F\left( {{x_* }} \right) + \frac{{1 - \mu }}{{2\lambda }}{\left\| {{x_{k + 1}} - {y_{k + 1}}} \right\|^2} + \frac{1}{{2\lambda }}{\left\| {{x_{k + 1}} - {x_k}} \right\|^2} \le F\left( {{x_k}} \right) - F\left( {{x_* }} \right) + \frac{{\alpha _k^2}}{{2\lambda }}{\left\| {{x_k} - {x_{k - 1}}} \right\|^2}.
	\end{equation*}
	Multiplying by $s_{k + 1}^2,$ we have 
	\begin{align}
	&{E_{k + 1}} + \left( {s_{k + 1}^2 - s_{k + 2}^2} \right)\left( {F\left( {{x_{k + 1}}} \right) - F\left( {{x_*}} \right)} \right) + \frac{{1 - \mu }}{{2\lambda }}s_{k + 1}^2{{\left\| {{x_{k + 1}} - {y_{k + 1}}} \right\|}^2} \nonumber \\
	&\le s_{k + 1}^2\left( {F\left( {{x_k}} \right) - F\left( {{x_*}} \right)} \right) + \frac{{{{\left( {{s_k} - 1} \right)}^2}}}{{2\lambda }}{{\left\| {{x_k} - {x_{k - 1}}} \right\|}^2} \nonumber\\
	&= s_{k + 1}^2\left( {F\left( {{x_k}} \right) - F\left( {{x_*}} \right)} \right) + \frac{{s_k^2}}{{2\lambda }}{{\left\| {{x_k} - {x_{k - 1}}} \right\|}^2} - \frac{{2{s_k} - 1}}{{2\lambda }}{{\left\| {{x_k} - {x_{k - 1}}} \right\|}^2} \le {E_k}. \nonumber
	\end{align}
	Then, combining with the Lemma \ref{l3}, we have
	\begin{equation}\label{C13}
	{E_{k + 1}} + \left( {\frac{{\left( {s_{k + 1}^2 - s_{k + 2}^2} \right)}}{{s_{k + 1}^2}} + \frac{{1 - \mu }}{{2{\tau _1}}}} \right)s_{k + 1}^2\left( {F\left( {{x_{k + 1}}} \right) - F\left( {{x_*}} \right)} \right) \le {E_k}.
	\end{equation} 
	Since that $\mathop {\lim \sup }\limits_{k \to \infty } \frac{{s_{k + 2}^2 - s_{k + 1}^2}}{{s_{k + 1}^2}} \le 0$ from Lemma \ref{L2}, we have $\frac{{s_{k + 1}^2 - s_{k + 2}^2}}{{s_{k + 1}^2}} \ge  - \frac{{1 - \mu }}{{4{\tau _1}}}$ for $k$ is large sufficiently, then, (\ref{C13}) can be deduce that for any $k \ge k_0,$ where $k_0$ is sufficiently large,
	\begin{equation}\label{C31}
	{E_{k + 1}} + \frac{{1 - \mu }}{{4{\tau _1}}}s_{k + 1}^2\left( {F\left( {{x_{k + 1}}} \right) - F\left( {{x_*}} \right)} \right) \le {E_k}, 
	\end{equation}
	i.e., $\sum\nolimits_{k = {k_0}}^{ + \infty } {s_{k + 1}^2\left( {F\left( {{x_{k + 1}}} \right) - F\left( {{x_*}} \right)} \right)}  <  + \infty .$ Since that increasing the finite term does not change the convergence of the series, we can easy to obtain that $\sum\limits_{k = 1}^\infty  {s_{k + 1}^2\left( {F\left( {{x_{k + 1}}} \right) - F\left( {{x_* }} \right)} \right)} $ is convergent. Hence, $F\left( {{x_{k + 1}}} \right) - F\left( {{x_* }} \right) = o\left( {\frac{1}{{s_{k + 1}^2}}} \right)$ holds ture.
	
	Further, since that $\left\{ {{E_k}} \right\}$ is convergent from (\ref{C31}), we have $\left\{ {s_{k + 1}^2{{\left\| {{x_{k + 1}} - {x_k}} \right\|}^2}} \right\}$ is bounded, which means that ${\left\| {{x_{k + 1}} - {x_k}} \right\|} \le O\left( {\frac{1}{{s_{k + 1}}}} \right),$ i.e., there exists a constant $c_1 >0$ such that
	$\left\| {{x_{k+1}} - {x_{k}}} \right\| \le \frac{c_1}{{{s_{k+1}}}}.$
	Recalling the assumption that $\sum\limits_{k = 1}^\infty  {\frac{1}{{s_{k+1}}}} $ is convergent, we can deduce that the sequence $\left\{ {{x_k}} \right\}$ is a Cauchy series. Suppose that $\mathop {\lim }\limits_{k \to \infty } {x_k} = \bar x,$ we conclude that $\left\{ {{x_k}} \right\}$ strongly converges to $\bar x \in {X^ * }$ since $F$ is lower semi-continuous convex. 
\end{proof}

\section{The sublinear convergence rates of IFB algorithms with special $t_k$}
In the following, we show the improved convergence rates for the IFBs with six special $t_k$ satisfying the Assumption 2.

\textbf{Case 1.} ${t_k} = {e^{{{\left( {k - 1} \right)}^\alpha }}},0 < \alpha  < 1.$  
\begin{corollary}\label{C3.5}
	Suppose that Assumption $A_1$ holds. Let $\left\{ {{x_k}} \right\}$ be generated by Algorithm 1 with $t_k$ in Case 1 and ${x_* } \in {X}.$ Then, we have \\
	1) $F\left( {{x_k}} \right) - F\left( {{x_*}} \right) = o\left( {\frac{1}{{{e^{2{{\left( {k - 1} \right)}^\alpha }}}}}} \right)$ and $ \left\| {{x_k} - {x_{k - 1}}} \right\| = O\left( {\frac{1}{{{e^{{{\left( {k - 1} \right)}^\alpha }}}}}} \right).$ \\
	2) $\left\{ {{x_k}} \right\}$ is converges sublinearly to $\bar x \in {X^ * }$ at the $O\left( {{{\left( {k - 1} \right)}^{\alpha \left\lceil {{\textstyle{1 \over \alpha }} - 1} \right\rceil }}{e^{ - {{\left( {k - 1} \right)}^\alpha }}}} \right)$ rate of convergence.
\end{corollary}
\begin{proof}
	We can easily verify that $\mathop {\lim }\limits_{k \to \infty } {k^{1 - \alpha }}\left( {\frac{{{t_{k + 1}}}}{{{t_k}}} - 1} \right) = \alpha,$ which means that Assumption $A_2$ holds. By setting ${s_k} = {t_k},$ we conclude that the result 1) is satisfied from Theorem \ref{L1}.
	It follows from the result 1) that there exists a positive constant \textcolor{red}{$c'$} such that  $\left\| {{x_k} - {x_{k - 1}}} \right\| \le \frac{{c'}}{{{e^{{{\left( {k - 1} \right)}^\alpha }}}}},$ we can deduce that  
	\[\forall p > 1,\quad \left\| {{x_{k + p}} - {x_k}} \right\| \le \sum\limits_{i = k + 1}^{k + p} {\left\| {{x_i} - {x_{i - 1}}} \right\|}  \le \sum\limits_{i = k + 1}^{k + p} {\frac{{c'}}{{{e^{{{\left( {k - 1} \right)}^\alpha }}}}}}  \le c'\int_k^{k + p} {{e^{ - {{\left( {x - 1} \right)}^\alpha }}}} dx.\]
	Since the convergence of $\int_1^{ + \infty } {{e^{ - {{\left( {x - 1} \right)}^\alpha }}}} dx,$ we see that $\sum\limits_{i = 1}^{ + \infty } {\left\| {{x_i} - {x_{i - 1}}} \right\|} $ is convergent, which means that ${\left\{ {{x_k}} \right\}}$ is a Cauchy series and converges strongly to $\bar x \in {X }.$
	Then, as $p \to \infty,$ we have
	\[{\left\| {{x_k} - \bar x} \right\| \le c'\int_k^{ + \infty } {{e^{ - {{\left( {x - 1} \right)}^\alpha }}}} dx = c'\int_{{{\left( {k - 1} \right)}^\alpha }}^{ + \infty } {\frac{1}{\alpha }{y^{\frac{1}{\alpha } - 1}}{e^{ - y}}} dy \le \frac{{c'}}{\alpha }\int_{{{\left( {k - 1} \right)}^\alpha }}^{ + \infty } {{y^{\left\lceil {\frac{1}{\alpha } - 1} \right\rceil }}{e^{ - y}}} dy.}\]
	Denote $\omega  = \left\lceil {{\textstyle{1 \over \alpha }} - 1} \right\rceil $ and $A = {\left( {k - 1} \right)^\alpha }.$  We can deduce that
	\begin{equation*}
	\left\| {{x_k} - \bar x} \right\| \le \frac{{c'}}{\alpha }\int_A^{ + \infty } {{y^\omega }{e^{ - y}}dy} = \frac{{c'}}{\alpha }{A^\omega }{e^{ - A}} + \frac{{c'}}{\alpha }\sum\limits_{j = 0}^{\omega  - 1} {\left( {\left( {\prod\limits_{i = 0}^j {\left( {\omega  - i} \right)} } \right)\left( {{A^{\omega  - j - 1}}{e^{ - A}}} \right)} \right)}  = O\left( {{A^\omega }{e^{ - A}}} \right),
	\end{equation*}
	which means that \[\left\| {{x_k} - \bar x} \right\| = O\left( {{{\left( {k - 1} \right)}^{\alpha \left\lceil {{\textstyle{1 \over \alpha }} - 1} \right\rceil }}{e^{ - {{\left( {k - 1} \right)}^\alpha }}}} \right).\]
	Hence, result 2) holds.
\end{proof}

\emph{Remark 4.} Notice that $\forall p > 1,\;{\left( {k - 1} \right)^{\alpha \left\lceil {{\textstyle{1 \over \alpha }} - 1} \right\rceil }}{e^{ - {{\left( {k - 1} \right)}^\alpha }}} = o\left( {\frac{1}{{{k^p}}}} \right),$ which means that the sublinear convergence rate of the IFB with the $t_k$ in Case 1 is faster than any order. 

\textbf{Case 2.} ${t_{k }} = \frac{{{k^r} - 1 + a}}{a}  \; \left( {r > 1}, a>0 \right).$ 
\begin{corollary}\label{C3.3}
	Suppose that Assumption $A_1$ holds. Let $\left\{ {{x_k}} \right\}$ be generated by Algorithm 1 with $t_k$ in Case 2 and ${x_* } \in {X }.$ Then, we have \\
	1) $F\left( {{x_k}} \right) - F\left( {{x_*}} \right) = o\left( {\frac{1}{{{k^{2r}}}}} \right) $ and $ \left\| {{x_k} - {x_{k - 1}}} \right\| = O\left( {\frac{1}{{{k^r}}}} \right).$ \\
	2) $\left\{ {{x_k}} \right\}$ is converges sublinearly to $\bar x \in {X^ * }$ at the $O\left( {\frac{1}{{{k^{r-1}}}}} \right)$ rate of convergence.
\end{corollary}
\begin{proof}
	It is easy to verify that
	$\mathop {\lim }\limits_{k \to \infty } k\left( {\frac{{{t_{k + 1}}}}{{{t_k}}} - 1} \right) = r$, which means that Assumption $A_2$ holds. Setting ${s_k} = {t_k},$ then, combining with $\mathop {\lim }\limits_{k \to \infty } \frac{{{s_k}}}{{{k^r}}} = \frac{1}{a}$ and $\sum\limits_{k = 1}^\infty  {\frac{1}{{{s_k}}}} $ is convergent, we can deduce that the result 1) holds and $\left\{ {{x_k}} \right\}$ converges strongly to $\bar x \in {X}$ by Theorem \ref{L1}.
	It follows from the result 1) that there exists a positive constant $c''$ such that $\left\| {{x_k} - {x_{k - 1}}} \right\| \le \frac{{c''}}{{{k^r}}}.$ Then, we can deduce that 
	\[\forall p > 1,\quad \left\| {{x_{k + p}} - {x_k}} \right\| \le \sum\limits_{i = k + 1}^{k + p} {\left\| {{x_i} - {x_{i - 1}}} \right\|}  \le c''\sum\limits_{i = k + 1}^{k + p} {\frac{1}{{{k^r}}}}  \le c''\int_k^{k + p} {\frac{1}{{{x^r}}}} dx.\]
	Then, \[\left\| {{x_k} - \bar x} \right\| \le \frac{{c''}}{{r - 1}}\frac{1}{{{k^{r - 1}}}},\; as\; p \to \infty.\]
	Hence, result 2) holds. 
\end{proof}

\emph{Remark 5.} For the $t_k$ in Case 2, we show that the convergence rate of function values and iterates related to the value of $r$. The larger $r,$ the better convergence rate Algorithm 1 achieves.

\textbf{Case 3.} ${t_k} = \frac{{{k^r} - 1 + a}}{a}\;\left( {r < 1,a > 0} \right).$
\begin{corollary}\label{C2.3}
	Suppose that Assumption $A_1$ holds. Let $\left\{ {{x_k}} \right\}$ be generated by Algorithm 1 with $t_k$ in Case 3 and ${x_* } \in {X }.$ Then, for any positive constant $p > 1,$\\
	1) $F\left( {{x_k}} \right) - F\left( {{x_*}} \right) = o\left( {\frac{1}{{{k^p}}}} \right)$ and $\left\| {{x_k} - {x_{k - 1}}} \right\| = O\left( {\frac{1}{{{k^p}}}} \right)$.\\
	2) $\left\{ {{x_k}} \right\}$ sublinearly converges to $\bar x \in {X^ * }$ at the $O\left( {\frac{1}{{{k^{p - 1}}}}} \right)$ rate of convergence. 
\end{corollary}
\begin{proof} 
	It is easy to verify that $\mathop {\lim }\limits_{k \to \infty } k\left( {\frac{{{t_{k + 1}}}}{{{t_k}}} - 1} \right) = r$, which means that Assumption $A_2$ holds.
	Since $\mathop {\lim }\limits_{k \to \infty } {k^r}\left( {{\gamma _k} - 1} \right) = \mathop {\lim }\limits_{k \to \infty } \frac{{ - {k^r}\left( {{t_{k + 1}} - {t_k} + 1} \right)}}{{{t_{k + 1}}}} = \mathop {\lim }\limits_{k \to \infty } \frac{{ - {k^r}\left( {{{\left( {k + 1} \right)}^r} - {k^r} + a} \right)}}{{{{\left( {k + 1} \right)}^r} - 1 + a}} =  - a,$
	we have 
	\begin{align}\label{C10}
	{\gamma _k} = 1 - \frac{a}{{{k^r}}} + o\left( {\frac{1}{{{k^r}}}} \right), \; as \; k \to + \infty.
	\end{align}
	For any positive constant $p > 1,$ denote that $s_1 = 1$ and ${s_k} = {\left( {k - 1} \right)^p},\; k \ge 2.$ Then, we have 
	\begin{align}\label{C11}
	{\alpha _k} = \frac{{{s_k} - 1}}{{{s_{k + 1}}}} = {\left( {1 - \frac{1}{k}} \right)^p} - \frac{1}{{{k^p}}} = 1 - \frac{p}{k} + o\left( {\frac{1}{k}} \right), \; as\; k \to + \infty.
	\end{align}
	By (\ref{C10}) and (\ref{C11}), we obtain that
	\begin{align}\label{C12}
	\mathop {\lim }\limits_{k \to \infty } {k^r}\left( {{\alpha _k} - {\gamma _k}} \right) = \mathop {\lim }\limits_{k \to \infty } {k^r}\left( { - \frac{p}{k} + \frac{a}{{{k^r}}} + o\left( {\frac{1}{{{k^r}}}} \right) + o\left( {\frac{1}{k}} \right)} \right) = a > 0,
	\end{align}
	which implies that ${{\alpha _k} \ge {\gamma _k}}$ for $k$ is large sufficiently. Hence, using ${s_k} \sim {k^p}$ and Theorem \ref{L1}, result 1) holds and $\left\{ {{x_k}} \right\}$ converges strongly to $\bar x \in {X }.$
	Similar with the proof of Corollary \ref{C3.3}, we conclude result 2).
\end{proof}

\textbf{Case 4.} ${t_1} = 1$ and ${t_k} = \frac{k}{{{{\ln }^\theta }k}}\left({\theta  > 0} \right),$  $\forall k \ge 2. $
\begin{corollary}\label{T3.3}
	Suppose that Assumption $A_1$ holds. Let $\left\{ {{x_k}} \right\}$ be generated by Algorithm 1 with $t_k$ in Case 4 and ${x_* } \in {X}.$ Then, for any positive constant $p \ge 2,$\\
	1) $F\left( {{x_k}} \right) - F\left( {{x_*}} \right) = o\left( {\frac{1}{{{k^{2p}}}}} \right)$ and $\left\| {{x_k} - {x_{k - 1}}} \right\| = O\left( {\frac{1}{{{k^p}}}} \right)$.\\
	2) $\left\{ {{x_k}} \right\}$ sublinearly converges to $\bar x \in {X^ * }$ at the $O\left( {\frac{1}{{{k^{p - 1}}}}} \right)$ rate of convergence. 
\end{corollary}
\begin{proof}
	We can prove that $\mathop {\lim }\limits_{k \to \infty } k\left( {\frac{{{t_{k + 1}}}}{{{t_k}}} - 1} \right) = 1,$  which means that Assumption $A_2$ holds.
	Observe that
	\[\begin{array}{l}
	{\gamma _k} = \frac{{{t_k}}}{{{t_{k + 1}}}} - \frac{1}{{{t_{k + 1}}}} = \left( {1 - \frac{1}{{k + 1}}} \right){\left( {\frac{{\ln \left( {k + 1} \right)}}{{\ln k}}} \right)^\theta } - \frac{{{{\ln }^\theta }\left( {k + 1} \right)}}{{k + 1}}\\
	= \left( {1 - \frac{1}{{k + 1}}} \right){\left( {1 + \frac{{\ln \left( {1 + \frac{1}{k}} \right)}}{{\ln k}}} \right)^\theta } - \frac{{{{\ln }^\theta }\left( {k + 1} \right)}}{{k + 1}}\\
	= \left( {1 - \frac{1}{{k + 1}}} \right)\left( {1 + \frac{{\theta \ln \left( {1 + \frac{1}{k}} \right)}}{{\ln k}} + o\left( {\frac{{\ln \left( {1 + \frac{1}{k}} \right)}}{{\ln k}}} \right)} \right) - \frac{{{{\ln }^\theta }\left( {k + 1} \right)}}{{k + 1}}\\
	= \left( {1 - \frac{1}{{k + 1}}} \right)\left( {1 + o\left( {\frac{1}{{k + 1}}} \right)} \right) - \frac{{{{\ln }^\theta }\left( {k + 1} \right)}}{{k + 1}}\\
	= 1 - \frac{1}{{k + 1}} - \frac{{{{\ln }^\theta }\left( {k + 1} \right)}}{{k + 1}} + o\left( {\frac{1}{{k + 1}}} \right).
	\end{array}\]
	Denote that ${s_k} = {k^p},$ where $p \ge 2.$ Then, we easily obtain that
	\[{\alpha _k} = \frac{{{s_k} - 1}}{{{s_{k + 1}}}} = {\left( {1 - \frac{1}{{k + 1}}} \right)^p} - \frac{1}{{{{\left( {k + 1} \right)}^p}}} = 1 - \frac{p}{{k + 1}} + o\left( {\frac{1}{{k + 1}}} \right).\]
	Hence, the condition ${{\alpha _k} \ge {\gamma _k}}$ holds true for $k$ is large sufficiently.
	Using Theorem \ref{L1}, result 1) holds. And similar with the proof of Corollary \ref{C3.3}, result 2) holds. 
\end{proof}

\textbf{Case 5.} $t_1 = 1$ and ${t_{k + 1}} = \frac{{1 + \sqrt {1 + 4t_k^2} }}{2},$ $\forall k \ge 1.$ 
\begin{corollary}\label{T3.1}
	Suppose that Assumption $A_1$ holds. Let $\left\{ {{x_k}} \right\}$ be generated by Algorithm 1 with $t_k$ in Case 5 and ${x_* } \in {X}.$ Then,\\
	1) $F\left( {{x_k}} \right) - F\left( {{x_*}} \right) = o\left( {\frac{1}{{{k^6}}}} \right)$ and $\left\| {{x_k} - {x_{k - 1}}} \right\| = O\left( {\frac{1}{{{k^3}}}} \right)$.\\
	2) $\left\{ {{x_k}} \right\}$ sublinearly converges to $\bar x \in {X^ * }$ at the $O\left( {\frac{1}{{{k^2}}}} \right)$ rate of convergence.
\end{corollary}
\begin{proof}
	We can easily obtain that $\mathop {\lim }\limits_{k \to \infty } k\left( {\frac{{{t_{k + 1}}}}{{{t_k}}} - 1} \right) = 1,$ which means that Assumption $A_2$ holds.
	Observe that
	\begin{align}\label{C5}
	{\gamma _k} = \frac{{{t_k} - 1}}{{{t_{k + 1}}}} = 1 - \frac{{{t_{k + 1}} - {t_k} - \frac{1}{2}}}{{{t_{k + 1}}}} - \frac{3}{{2{t_{k + 1}}}}. 
	\end{align}
	Since that $\mathop {{\rm{lim}}}\limits_{k \to \infty } {t_k}\left( {{t_{k + 1}} - {t_k} - \frac{1}{2}} \right) = \mathop {{\rm{lim}}}\limits_{k \to \infty } {t_k}\left( {\frac{{1 + \sqrt {1 + 4t_k^2} }}{2} - {t_k} - \frac{1}{2}} \right) = \frac{1}{8}$ and $\mathop {\rm{lim}}\limits_{k \to \infty } \frac{{{t_k}}}{k} = \frac{1}{2},$ we can deduce that 
	\[\mathop {{\rm{lim}}}\limits_{k \to \infty } {k^2}\left( {\frac{{{t_{k + 1}} - {t_k} - \frac{1}{2}}}{{{t_{k + 1}}}}} \right) = \frac{1}{2},\]
	which means that 
	\begin{align}\label{C6}
	\frac{{{t_{k + 1}} - {t_k} - \frac{1}{2}}}{{{t_{k + 1}}}} = \frac{1}{{2{k^2}}} + o\left( {\frac{1}{{{k^2}}}} \right),\; as \; k \to  + \infty.
	\end{align}
	By Stolz theorem, we obtain \[\mathop {\rm{lim}}\limits_{k \to \infty } \frac{{\frac{1}{2}k - {t_{k + 1}}}}{{\ln k}} = \mathop {\rm{lim}}\limits_{k \to \infty } \frac{{\frac{1}{2} - \left( {{t_{k + 2}} - {t_{k + 1}}} \right)}}{{\ln \left( {1 + \frac{1}{k}} \right)}} = \mathop {\rm{lim}}\limits_{k \to \infty } \frac{{ - k}}{{{t_{k + 1}}}}{t_{k + 1}}\left( {{t_{k + 2}} - {t_{k + 1}} - \frac{1}{2}} \right) =  - \frac{1}{4},\]
	then, 
	\[\mathop {\lim }\limits_{k \to \infty } \frac{{\frac{3}{{2{t_{k + 1}}}} - \frac{3}{k}}}{{\frac{{\ln k}}{{{k^2}}}}} = \mathop {\lim }\limits_{k \to \infty } 3\frac{k}{{{t_{k + 1}}}}\left( {\frac{{\frac{1}{2}k - {t_{k + 1}}}}{{\ln k}}} \right) =  - \frac{3}{2},\]
	which means that
	\begin{align}\label{C7}
	\frac{3}{{2{t_{k + 1}}}} =   \frac{3}{k} - \frac{3}{2}\frac{{\ln k}}{{{k^2}}} + o\left( {\frac{{\ln k}}{{{k^2}}}} \right), \; as \; k \to  + \infty.
	\end{align} 
	Hence, by (\ref{C5})--(\ref{C7}), we have
	\begin{align}\label{C8}
	{\gamma _k} = 1 - \frac{3}{k} + \frac{{3\ln k}}{{2{k^2}}} + o\left( {\frac{{\ln k}}{{{k^2}}}} \right),\; as \; k \to  + \infty.
	\end{align}
	Denote that ${s_1} = {s_2} = 1$ and ${s_k} = \frac{{{{\left( {k - 1} \right)}^3}}}{{{{\left( {\int_1^{k - 1} {\frac{{\ln x}}{{{x^2}}}dx} } \right)}^2}}},\; \forall k \ge 3.$ Then, since that $\int_1^{ + \infty } {\frac{{\ln x}}{{{x^2}}}dx}  = 1$ and $\frac{{{{\left( {\int_{k - 1}^k {\frac{{\ln x}}{{{x^2}}}dx} } \right)}^2}}}{{{k^3}}} = o\left( {\frac{{\ln k}}{{{k^2}}}} \right),$ we have 
	\begin{align}
	{\alpha _k} &= \frac{{{s_k} - 1}}{{{s_{k + 1}}}} \label{C9} \\
	&= {{\left( {1 - \frac{1}{k}} \right)}^3}{{\left( {1 + \frac{{\int_{k - 1}^k {\frac{{\ln x}}{{{x^2}}}dx} }}{{\int_1^{k - 1} {\frac{{\ln x}}{{{x^2}}}dx} }}} \right)}^2} - \frac{{{{\left( {\int_{k - 1}^k {\frac{{\ln x}}{{{x^2}}}dx} } \right)}^2}}}{{{k^3}}} \nonumber \\
	&\ge {{\left( {1 - \frac{1}{k}} \right)}^3}{{\left( {1 + \frac{{\frac{{\ln k}}{{{k^2}}}}}{{\int_1^{ + \infty } {\frac{{\ln x}}{{{x^2}}}dx} }}} \right)}^2} - \frac{{{{\left( {\int_{k - 1}^k {\frac{{\ln x}}{{{x^2}}}dx} } \right)}^2}}}{{{k^3}}}\nonumber \\
	&= {{\left( {1 - \frac{1}{k}} \right)}^3}{{\left( {1 + \frac{{\ln k}}{{{k^2}}}} \right)}^2} - \frac{{{{\left( {\int_{k - 1}^k {\frac{{\ln x}}{{{x^2}}}dx} } \right)}^2}}}{{{k^3}}} \nonumber\\
	&= 1 - \frac{3}{k} + \frac{{2\ln k}}{{{k^2}}} + o\left( {\frac{{\ln k}}{{{k^2}}}} \right),\;as\;k \to  + \infty. \nonumber
	\end{align}
	Obviously, by (\ref{C8}) and (\ref{C9}), we have $\mathop {\lim }\limits_{k \to \infty } \frac{{{\alpha _k} - {\gamma _k}}}{{\frac{{\ln k}}{{{k^2}}}}} \ge \frac{1}{2}.$ Therefore, $\alpha _k \ge \gamma _k$ for $k$ is large sufficiently.
	Using the fact that ${s_k} \sim {k^3}$ and Theorem \ref{L1}, we conclude the result 1) and  $\left\{ {{x_k}} \right\}$ converges strongly to $\bar x \in {X}.$
	Further, similar with the proof of Corollary \ref{C3.3}, result 2) holds.
\end{proof}

\textbf{Case 6.} ${t_{k}} = \frac{{k - 1 + a}}{a}  \; \left( {a > 0} \right).$
\begin{corollary}\label{T3.2}
	Suppose that Assumption $A_1$ holds. Let $\left\{ {{x_k}} \right\}$ be generated by Algorithm 1 with $t_k$ in Case 6 and ${x_* } \in {X}.$ Then,\\
	1) $F\left( {{x_k}} \right) - F\left( {{x_*}} \right) = o\left( {\frac{1}{{{k^{2\left( {a + 1} \right)}}}}} \right)$ and $\left\| {{x_k} - {x_{k - 1}}} \right\| = O\left( {\frac{1}{{{k^{a + 1}}}}} \right)$.\\
	2) $\left\{ {{x_k}} \right\}$ sublinearly converges to $\bar x \in {X^ * }$ at the $O\left( {\frac{1}{{{k^a}}}} \right)$ rate of convergence. 
\end{corollary}
\begin{proof}
	It is easy to verify that $\mathop {\lim }\limits_{k \to \infty } k\left( {\frac{{{t_{k + 1}}}}{{{t_k}}} - 1} \right) = 1$, which means that Assumption $A_2$ holds. 
	Observe that ${\gamma _k} = \frac{{{t_k} - 1}}{{{t_{k + 1}}}} = \frac{{k - 1}}{{k + a}}.$ 
	For $a \ge 1,$ denote ${s_k} = {\left( {k + a - 1} \right)^{a + 1}}.$ Otherwise, denote ${s_1} = {s_2} = 1,$ and ${s_k} = \frac{{{{\left( {k + a - 1} \right)}^{a + 1}}}}{{\int_1^{k - 1} {\frac{{\ln x}}{{{x^{1 + a}}}}dx} }},\; \forall k \ge 3.$
	
	1) For the case $a>1,$ we have 
	\begin{align}\label{C2}
	\begin{array}{l}
	{\alpha _k} = \frac{{{{\left( {k + a - 1} \right)}^{a + 1}} - 1}}{{{{\left( {k + a} \right)}^{a + 1}}}} = {\left( {1 - \frac{1}{{k + a}}} \right)^{a + 1}} - \frac{1}{{{{\left( {k + a} \right)}^{a + 1}}}}\\
	\quad \;\; = 1 - \frac{{a + 1}}{{k + a}} + \frac{{a\left( {a + 1} \right)}}{2}\frac{1}{{{{\left( {k + a} \right)}^2}}} + o\left( {\frac{1}{{{{\left( {k + a} \right)}^2}}}} \right) - \frac{1}{{{{\left( {k + a} \right)}^{a + 1}}}}\\
	\quad \;\; = {\gamma _k} + \frac{{a\left( {a + 1} \right)}}{2}\frac{1}{{{{\left( {k + a} \right)}^2}}} + o\left( {\frac{1}{{{{\left( {k + a} \right)}^2}}}} \right), as \; k \to  + \infty.
	\end{array}
	\end{align}
	which means that $\mathop {\lim }\limits_{k \to \infty } {\left( {k + a} \right)^2}\left( {{\alpha _k} - {\gamma _k}} \right) = \frac{{a\left( {a + 1} \right)}}{2} > 0,$ i.e., ${{\alpha _k} > {\gamma _k}}$ for $k$ is large sufficiently.
	
	2) For the case $a=1,$ we have 
	\begin{align}\label{C3}
	{\alpha _k} = \frac{{{k^2} - 1}}{{{{\left( {k + 1} \right)}^2}}} = {\left( {1 - \frac{1}{{k + 1}}} \right)^2} - \frac{1}{{{{\left( {k + 1} \right)}^2}}} = {\gamma _k}.
	\end{align} 
	Obviously, ${{\alpha _k} \ge {\gamma _k}}$ for any $k \ge 1.$
	
	3) For the case $a<1,$ we have 
	\begin{align}
	{\alpha _k} &= \frac{{{s_k} - 1}}{{{s_{k + 1}}}} = {{\left( {1 - \frac{1}{{k + a}}} \right)}^{a + 1}}\left( {1 + \frac{{\int_{k - 1}^k {\frac{{\ln x}}{{{x^{1 + a}}}}dx} }}{{\int_1^{k - 1} {\frac{{\ln x}}{{{x^{1 + a}}}}dx} }}} \right) - \frac{{\int_1^k {\frac{{\ln x}}{{{x^{1 + a}}}}dx} }}{{{{\left( {k + a} \right)}^{a + 1}}}} \label{C1} \\
	&\ge {{\left( {1 - \frac{1}{{k + a}}} \right)}^{a + 1}}\left( {1 + \frac{{\frac{{\ln k}}{{{k^{1 + a}}}}}}{{\int_1^{ + \infty } {\frac{{\ln x}}{{{x^{1 + a}}}}dx} }}} \right) - \frac{{\int_1^k {\frac{{\ln x}}{{{x^{1 + a}}}}dx} }}{{{{\left( {k + a} \right)}^{a + 1}}}} \nonumber \\
	&= {{\left( {1 - \frac{1}{{k + a}}} \right)}^{a + 1}}\left( {1 + \frac{{{a^2}\ln k}}{{{k^{1 + a}}}}} \right) - \frac{{\int_1^k {\frac{{\ln x}}{{{x^{1 + a}}}}dx} }}{{{{\left( {k + a} \right)}^{a + 1}}}},\;as\;k \to  + \infty.\nonumber
	\end{align}
	From ${\left( {1 - \frac{1}{{k + a}}} \right)^{a + 1}} = 1 - \frac{{a + 1}}{{k + a}} + \frac{{a\left( {a + 1} \right)}}{2}\frac{1}{{{{\left( {k + a} \right)}^2}}} + o\left( {\frac{1}{{{{\left( {k + a} \right)}^2}}}} \right),$ $\frac{1}{{{{\left( {k + a} \right)}^2}}} = o\left( {\frac{{\ln k}}{{{k^{1 + a}}}}} \right)$ and $\frac{{\int_1^k {\frac{{\ln x}}{{{x^{1 + a}}}}dx} }}{{{{\left( {k + a} \right)}^{a + 1}}}} = o\left( {\frac{{\ln k}}{{{k^{1 + a}}}}} \right)$, (\ref{C1}) can be deduced that 
	\[{\alpha _k} \ge {\gamma _k} + \frac{{k - 1}}{{k + a}}\frac{{{a^2}\ln k}}{{{k^{1 + a}}}} + o\left( {\frac{{\ln k}}{{{k^{1 + a}}}}} \right),\]
	which implies that $\mathop {\lim }\limits_{k \to \infty } \frac{{{\alpha _k} - {\gamma _k}}}{{\frac{{\ln k}}{{{k^{1 + a}}}}}} \ge {a^2} > 0,$ i.e., ${\alpha _k} \ge {\gamma _k}$ for $k$ is large sufficiently.
	
	Hence, $\gamma _k \le \alpha _k$ holds for $k$ is large sufficiently. Since that ${s_k} = O\left( {{k^{a + 1}}} \right),$ we conclude the result 1) from Theorem \ref{L1} and $\left\{ {{x_k}} \right\}$ converges strongly to $\bar x \in {X }.$
	Further, similar with the proof of Corollary \ref{C3.3}, result 2) holds. 
\end{proof}

\emph{Remark 6.} We see that $t_k$ in case 2 is the $t_k$ proposed in FISTA\_CD \citep[see,][]{Convea} but with a wider scope of $a.$ Corollary \ref{T3.2} shows that the convergence rates of IFB with $t_k$ in Case 2 related to the value of $a.$

Notice that both of the convergence results in Corollary \ref{C3.5} and Corollary \ref{C2.3} enjoy sublinear convergence rate of $o\left( {\frac{1}{{{k^p}}}} \right)$ for any $p>1.$ Here,
We give a further analysis for the convergence rate of the IFB from another aspect. From our Assumption $A_2,$ we can derive that 
\[{\gamma _k} = 1 - \frac{c}{{{k^\sigma}}} + o\left( {\frac{1}{{{k^\sigma}}}} \right) - \frac{1}{{{t_{k + 1}}}}.\] 
For ${t_k}$ in Case 1, we have ${\gamma _k} = 1 - \frac{\alpha}{{{k^{1 - \alpha }}}} + o\left( {\frac{1}{{{k^{1 - \alpha }}}}} \right);$ For ${t_{k}}$ in Case 3, we have corresponding ${\gamma _k} = 1 - \frac{a}{{{k^r}}} + o\left( {\frac{1}{{{k^r}}}} \right).$ 
Obviously, these two ${\gamma _k}$ are of the similar magnitude, in particular, they should be of the same order if we choose $r=0.5,$ $a=0.5,$ and $\alpha=0.5,$ theoretically. Thus, it's reasonable that the corresponding IFBs have similar numerical experiments. Numerical results in Section 4 can confirm this conclusion.

\section{Inertial Forward-Backward Algorithm with an Adaptive Modification}
 
For solving the problem ($P$), the authors in \citet{wenbo} showed that under the error bound condition, the sequences $\left\{ {{x_k}} \right\}$ and $\left\{ {F\left( {{x_k}} \right)} \right\}$ generated by FISTA with fixed restart are $R$-linearly convergent. In \citet{Restart}, an adaptive scheme for FISTA were proposed and enjoyed global linear convergence of the objective values when applying this method to problem ($P$) with $f$ being strongly convex and $g = 0.$ And the authors stated that after a certain number of iterations, adaptive restarting may provide linear convergence for Lasso, while they didn't prove similar results for general nonsmooth convex problem ($P$). In this section, we will explain that Inertial forward-backward algorithm with an adaptive modification enjoys same convergence results as we proved in Section 2 and Section 3. The adaptive modification scheme is described below:
\begin{algorithm}[tbhp]
	\caption{ Inertial forward-backward algorithm with an adaptive modification (IFB\_AdapM) }		
	\hspace*{0.2cm} \textbf{Step 0.} Take ${y_1} = {x_0} \in {R^n},{t_1} = 1.$ Input $\lambda  = \frac{\mu }{{{L_f}}},$ where ${\mu } \in \left] {{\rm{0}},{\rm{1}}} \right[$.\\
	\hspace*{0.7cm} \textbf{Step k.}  Compute \\
	\hspace*{3cm} ${x_k} = {T_\lambda }\left( {{y_k}} \right) = {\rm{pro}}{{\rm{x}}_{\lambda g}}\left( {{y_k} - \lambda \nabla f\left( {{y_k}} \right)} \right)$\\
	\hspace*{3.3cm} ${y_{k + 1}} = {x_k} + {\gamma _k}\left( {{x_k} - {x_{k - 1}}} \right),$\\
	\hspace*{1.95cm}where $\left\{ \begin{array}{l}
		{\gamma _k} = 0,\;{\rm{if}}\;{\left( {{y_k} - {x_k}} \right)^T}\left( {{x_k} - {x_{k - 1}}} \right) > 0\;{\rm{or}}\;F\left( {{x_k}} \right) > F\left( {{x_{k - 1}}} \right),\\
		{\gamma _k} = \frac{{{t_k} - 1}}{{{t_{k + 1}}}},\;{\rm{otherwise}.}
	\end{array} \right.$	
\end{algorithm}

Note that the adaptive modification condition is same with the adpative restart scheme in \citet{Restart}. Here, we call the condition ${\left( {{y_k} - {x_k}} \right)^T}\left( {{x_k} - {x_{k - 1}}} \right) > 0$ as gradient modification scheme and call $F\left( {{x_k}} \right) > F\left( {{x_{k - 1}}} \right)$ as function modification scheme. While, unlike the restart strategy setting ${t_k} = 0$ every time the restart condition holds to make the update of mumentum restarts from 0, Algorithm 2 sets the momentum back to 0 (Called adaptive modification step) at the current iteration but don't interrupt the update of mumentum. Based on Theorem 2.2 and the fact ${\gamma_k} = 0 \le {\alpha _k} = \frac{{{s_k} - 1}}{{{s_{k + 1}}}},$ we can obtain the same convergence rates for the function values and iterates of Algorithm 2. Specifically, Algorithm 2 with ${t_k} = {e^{{{\left( {k - 1} \right)}^\alpha }}}\left( {0 < \alpha  < 1} \right)$ or ${t_k} = \frac{{{k^r} - 1 + a}}{a}\;\left( {r < 1,a > 0} \right)$ converges with any sublinear rate of type $\frac{1}{{{k^p}}}$ and the corresponding numerical performances compare favourably with FISTA equipped with the fixed restart scheme or both the fixed and adaptive restart schemes, which has $R$-linearly convergence rate (See the numerical experiments in Section 5).

\section{Numerical Experiments}

In this section, we conduct numerical experiments to study the numerical performance of IFB with different options of $t_k$ and to verify our theoretical results.
The codes are available at \url{https://github.com/TingWang7640/Paper_EB.git}

\textbf{LASSO} We first consider the LASSO
\begin{align}\label{O49}
\mathop {\min }\limits_{x \in {R^n}} F\left( x \right) = \frac{1}{2}{\left\| {Ax - b} \right\|^2} + \delta  {\left\| x \right\|_1}.
\end{align}

We generate an $A \in {R^{m \times n}}$ be a Gaussian matrix and randomly generate
a $s-$sparse vector $\hat x$ and set $b = A\hat x + 0.5\varepsilon ,$ where $\varepsilon$ has standard i.i.d. Gaussian entries. And set $\delta  = 1.$ We observe that (\ref{O49}) is in the form of problem ($P$) with $f\left( x \right) = \frac{1}{2}{\left\| {Ax - b} \right\|^2}$ and $g\left( x \right) = \delta {\left\| x \right\|_1}.$ It is clear that $f$ has a Lipschitz continuous gradient and ${L_f} = {\lambda _{\max }}\left( {{A^T}A} \right).$ Moreover, we can observe that (\ref{O49}) satisfies the local error bound condition based on the third example in Introduction with $h\left( x \right) = \frac{1}{2}{\left\| x \right\|^2}$ and $c =  - {A^T}b.$
We terminate the algorithms once $ \left\| {\partial F\left( {{x_k}} \right)} \right\| < {10^{ - 8}}  .$ 

Considering Corollary \ref{C3.3}. We know that in theory, the rate of convergence of IFB with $t_k$ in Case 2 should improve constantly as $r$ increasing. In the Fig.1, we test four choices of $r,$ which is $r=2,$ $r=4,$ $r=6$ and $r=8,$ to show the same result in experiments as in theory. Denote that the IFB with ${t_{k }} = \frac{{{k^r} - 1 + a}}{a} $ is called as ``FISTA\_pow(r)''. Here we set $a=4.$ And the constant stepsize is $\lambda = \frac{{0.98}}{{{L_f}}}.$ 
\begin{figure}[H]
	\centering
	\subfloat
	{
		\includegraphics[width=6.0cm]{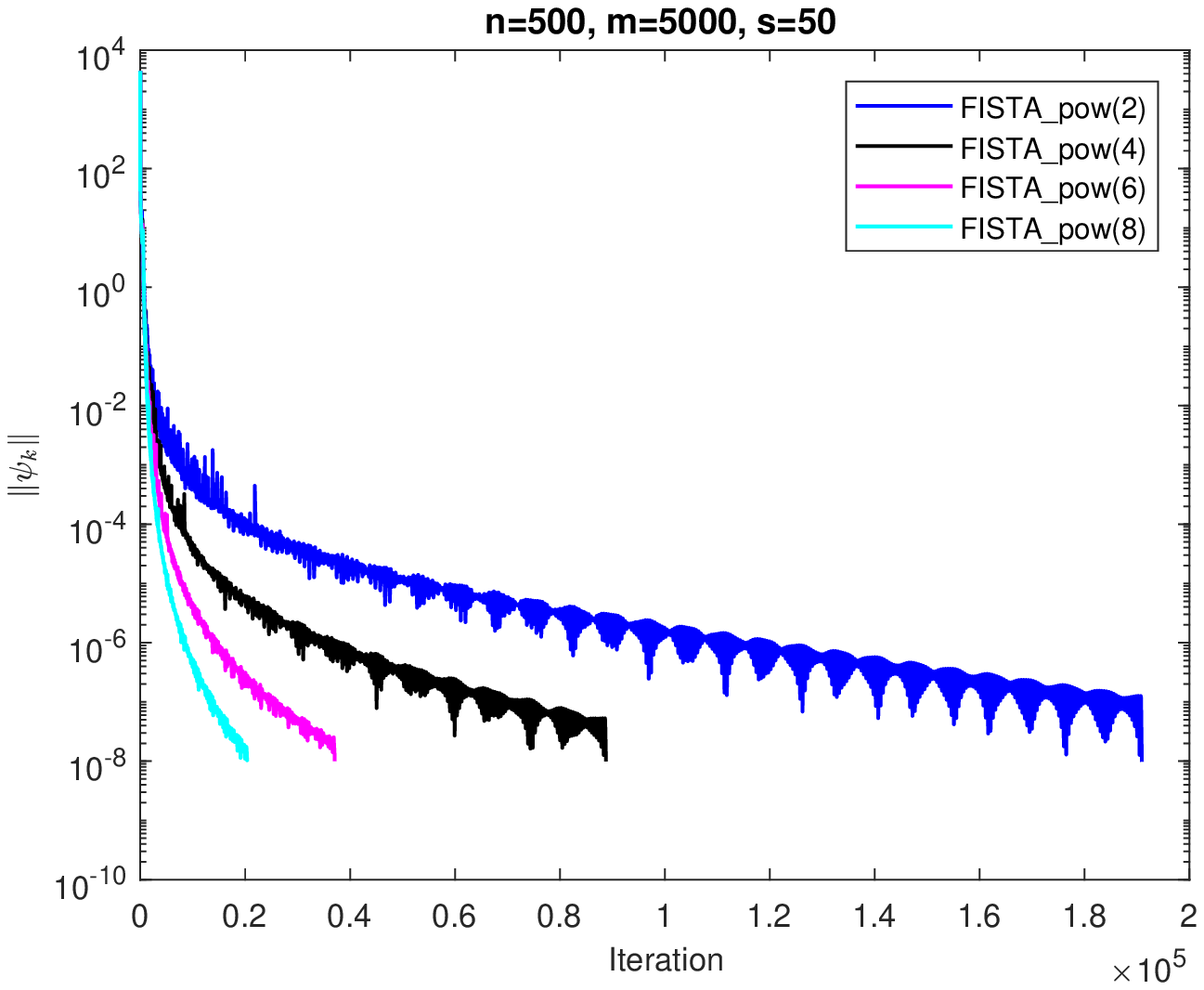}	
	}
	\subfloat
	{
		\includegraphics[width=6.0cm]{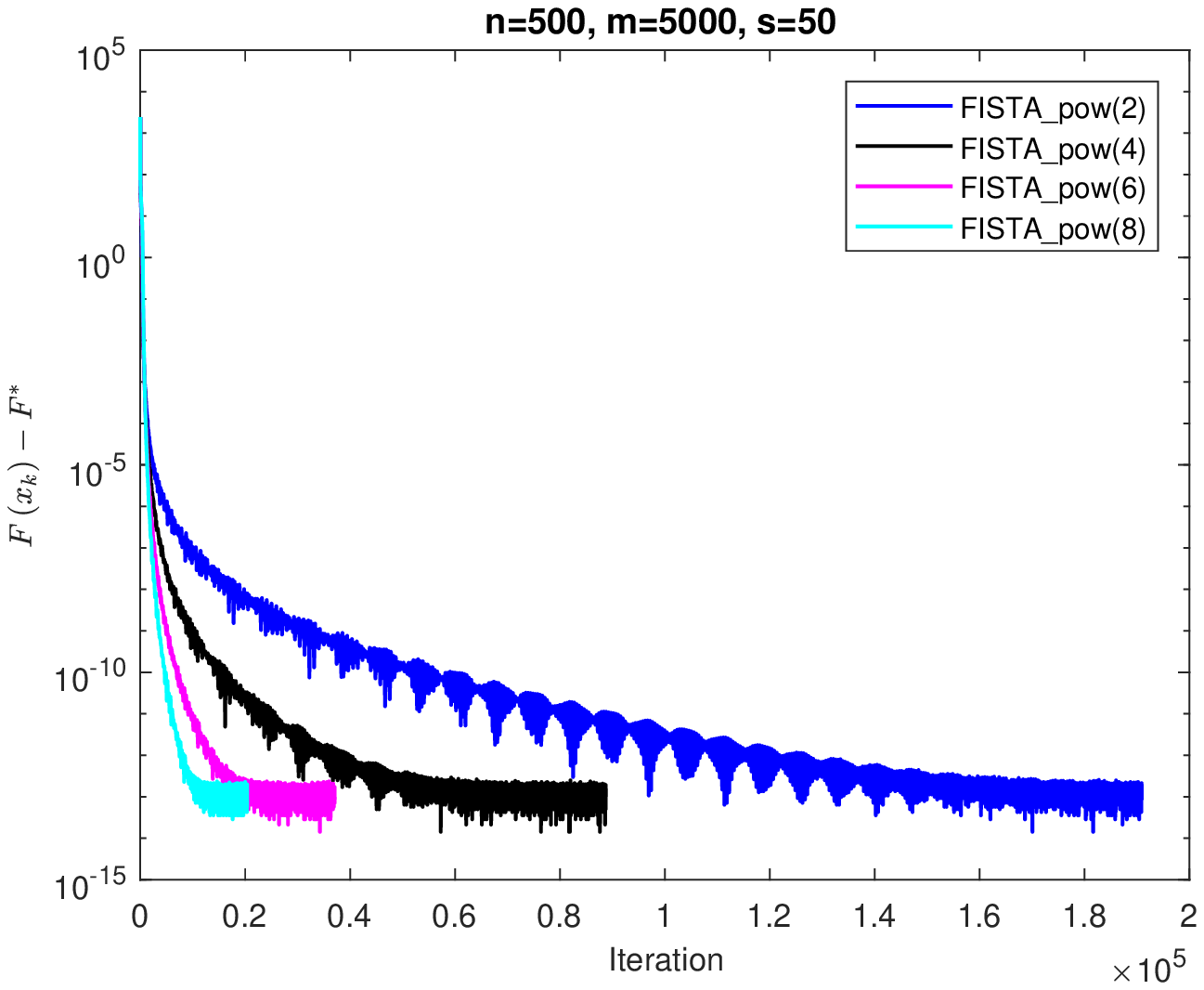}	
	}
	\caption{Computational results for the convergence of $\left\| {{\psi _k}} \right\| $ and $\left( {F\left( {{x_k}} \right) - F^*} \right).$}
	\label{fig:1}	
\end{figure}

In Corollary \ref{T3.2}, we show that the convergence rate of corresponding IFB greatly related to the value of $a.$ In the Fig.2, we test four choices of $a,$ which is $a=4,$ $a=6,$ $a=8$ and $a=10,$ to verify our theoretical results. Set $\lambda  = \frac{0.98}{{{L_f}}}.$
\begin{figure}[H]
	\centering
	\subfloat
	{
		\includegraphics[width=6.0cm]{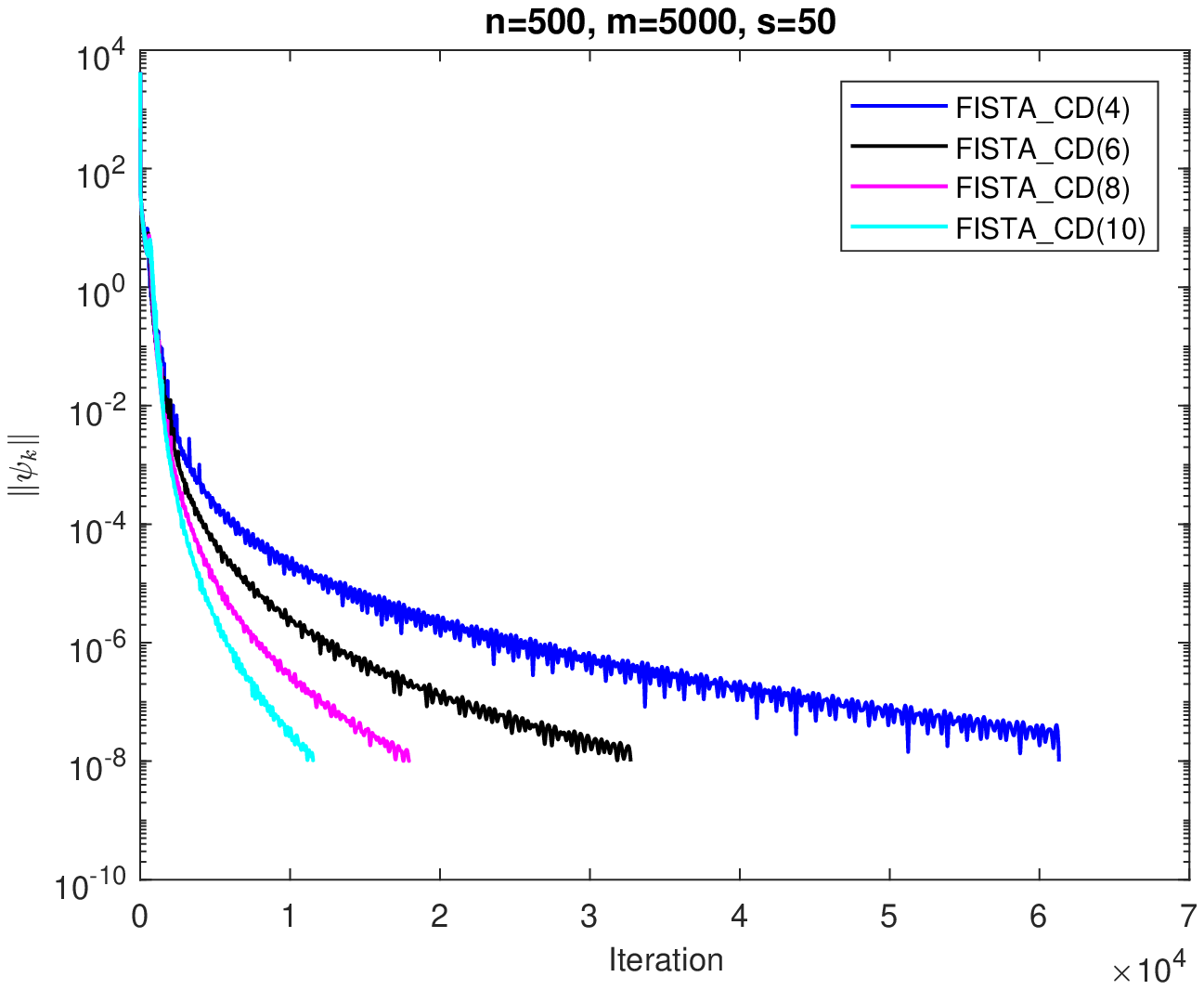}	
	}
	\subfloat
	{
		\includegraphics[width=6.0cm]{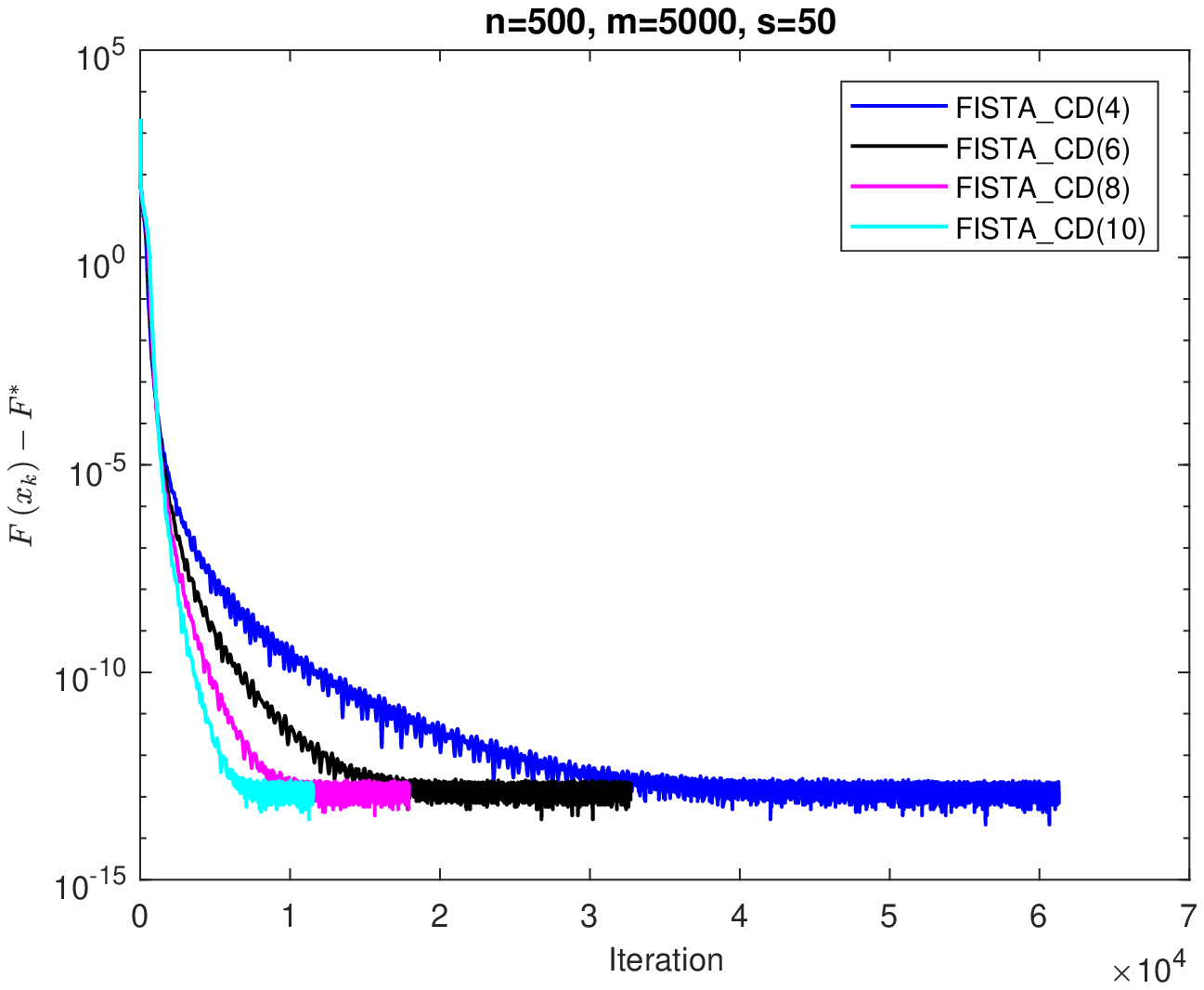}	
	}
	\caption{Computational results for the convergence of $\left\| {{\psi _k}} \right\| $ and $\left( {F\left( {{x_k}} \right) - F^*} \right).$}
	\label{fig:2}	
\end{figure}

Now, we perform numerical experiments to study the IFB with five choices of $t_k.$ Notice that the IFBs with $t_k$ discussed in Case 1 and Case 3 enjoy the rates of convergence better than any order of convergence rate, and in the end of last section, we emphasize that these two IFBs should achieve almost the same numerical experiments if we set the related parameters as $r=0.5,$ $a=0.5,$ and $\alpha=0.5.$ Hence, we consider the following algorithms:\\
1) FISTA; \\
2) FISTA\_CD with $a = 4$; \\
3) FISTA\_pow(8), i.e., the IFB with ${t_{k }} = \frac{{{k^r} - 1 + a}}{a}  \; \left( {r = 8 \; {\rm{and}} \; a = 4} \right)$.\\
4) FISTA\_pow(0.5), i.e., the IFB with ${t_{k }} = \frac{{{k^r} - 1 + a}}{a}  \; \left( {r = 0.5 \; {\rm{and}} \; a = 0.5} \right)$.\\
5) FISTA\_exp, i.e., the IFB with ${t_k} = {e^{{{\left( {k - 1} \right)}^\alpha }}},0 < \alpha  < 1.$ And set $\alpha = 0.5.$

Set $\lambda = \frac{{0.98}}{{{L_f}}}.$ 
\begin{figure}[tbhp]
	\centering
	\subfloat
	{
		\includegraphics[width=6.0cm]{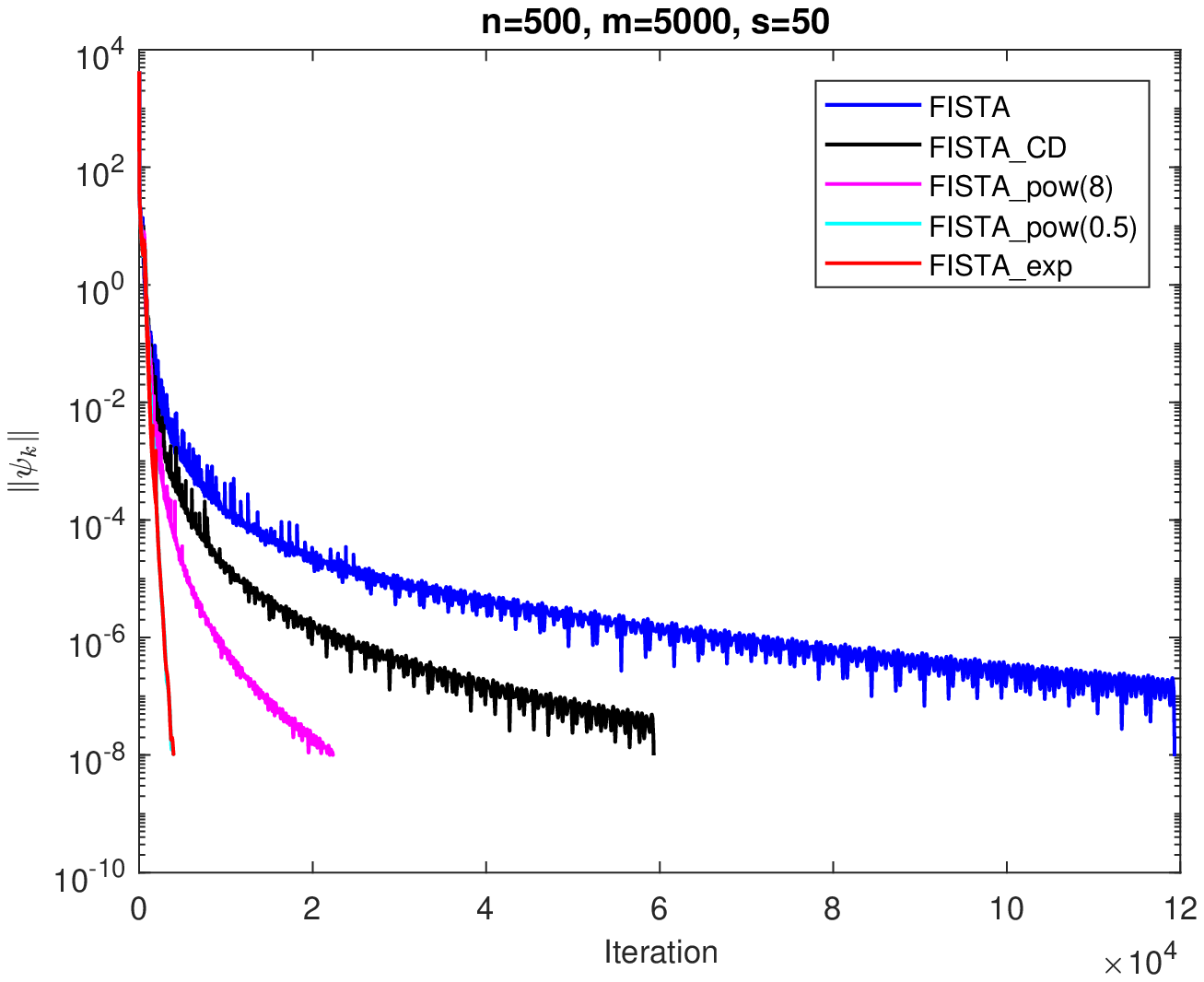}	
	}
	\subfloat
	{
		\includegraphics[width=6.0cm]{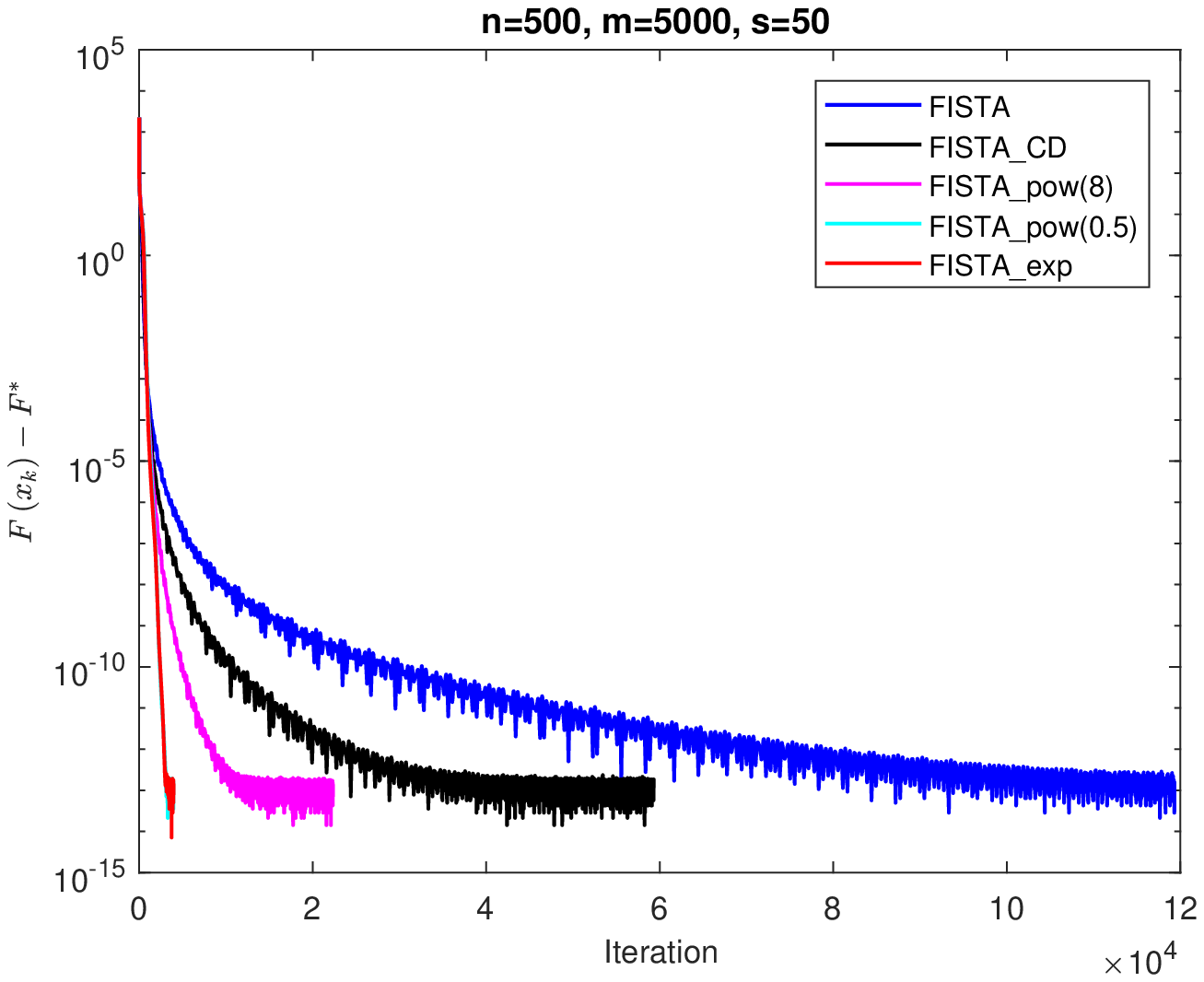}	
	}
	\caption{Computational results for the convergence of $\left\| {{\psi _k}} \right\| $ and $\left( {F\left( {{x_k}} \right) - F^*} \right).$}
	\label{fig:3}	
\end{figure}
Our computational results are presented in Fig.3.
We see that FISTA\_exp and FISTA\_pow(0.5) cost many fewer steps than FISTA\_CD and FISTA, and faster than FISTA\_pow(8). This results are same as the theoretical analyses in Section 3. And we see that the two lines of FISTA\_exp and FISTA\_pow(0.5) almost coincide, here, we give the detail number of iterations: for FISTA\_exp, it's number of iteration is 3948, and for FISTA\_pow(0.5), it's number of iteration is 3964, which \textcolor{red}{confirms} our theoretical analysis.  

\textbf{Sparse Logistic Regression.} We also consider the sparse logistic regression with the $l_1$ regularized, that is 
\begin{align}\label{O50}
\mathop {\min }\limits_x \frac{1}{n}\sum\limits_{i = 1}^n {\log \left( {1 + \exp \left( { - {l_i}\left\langle {{h_i},x} \right\rangle } \right)} \right)}  + \delta \left\| {{x}} \right\|_1,
\end{align}
where ${h_i} \in {R^m},$ ${l_i} \in \left\{ { - 1,1} \right\},i = 1, \cdots n.$
Define ${K_{ij}} =  - {l_i}{h_{ij}}$ and ${L_f} = \frac{4}{n}\left\| {{K^T}K} \right\|.$ It satisfies the local error bound condition since the third example in Introduction with $h\left( x \right) = \frac{1}{n}\sum\limits_{i = 1}^n {\log \left( {1 + \exp \left( {{x_i}} \right)} \right)} $ and $A=K,$ $c = 0.$
Set $\delta  = 1.e - 2.$ We take three datasets ``w4a'', ``a9a'' and ``sonar'' from LIBSVM \citep[see,][]{LIBSVM}. And the computational results relative to the number of iterations are reported in following Table 1.
\begin{table}[H]
	\caption{Comparison of the number of iterations}
	\centering
	\begin{tabular}{cccccc}
		\hline
		& FISTA & FISTA\_CD & FISTA\_pow(8) & FISTA\_pow(0.5) & FISTA\_exp \\ \hline
		`` w4a ''   & 1147  & 760       & 544           & 510             & 548        \\ \hline
		`` a9a ''   & 2049  & 1289      & 757           & 623             & 714        \\ \hline
		`` sonar '' & 8405  & 3406      & 1586          & 922             & 980        \\ \hline
	\end{tabular}
\end{table}
We see from Table 1 that FISTA\_exp, FISTA\_pow(0.5) and FISTA\_pow(8) outperform FISTA and FISTA\_CD and the numerical results are consistent with the theoretical ones. 

\textbf{Strong convex quadratic programming with simplex constraints.}
\[\mathop {\min }\limits_{x \in \left[ {sl,su} \right]} \frac{1}{2}{x^T}Ax + {b^T}x,\]
where $A \in {R^{m \times m}}$ is a symmetric positive definite matrix  generated by $A = {B^T}B + sI$ where $B \in {R^{\frac{m}{2} \times m}}$ with i.i.d. standard Gaussian entries and $s$ chosen uniformly at random from $\left[ {0,1} \right]$. The vector $b \in {R^m}$ is generated with i.i.d. standard Gaussian entries. Set $su$ = ones(m,1) and $sl$ = -ones(m,1). Notice that $f\left( x \right) = \frac{1}{2}{x^T}Ax + {b^T}x$ with ${\mu _f} = {\lambda _{\min }}\left( A \right)$ and ${L_f} = {\lambda _{\max }}\left( A \right),$ and $g\left( x \right) = {\delta _{\left[ {sl,su} \right]}}\left( x \right).$ Here, we terminate the algorithms once $ \left\| {\partial F\left( {{x_k}} \right)} \right\| < {10^{ - 6}}  .$ 

Now we perform numerical experiments with (1) Forward-Backward method without inertial (FB); (2) FISTA with fixed and adaptive restart schemes (FISTA\_R); (3) IFB with ${\beta ^ * } := \frac{{\sqrt {{\mu _f}}  - \sqrt {{L_f}} }}{{\sqrt {{\mu _f}}  + \sqrt {{L_f}} }}$ (IFB\_${\beta ^ * }$); (4) FISTA\_exp; (5) Algorithm 2 (Gradient scheme) with ${t_k} = {e^{\sqrt {k - 1} }}$ (IFB\_AdapM\_exp).

According to Corollary 3.1, we know that $o\left( {\frac{1}{{{k^p}}}} \right)$ sublinearly convergence rate for the sequences $\left\{ {{x_k}} \right\}$ and $\left\{ {F\left( {{x_k}} \right)} \right\}$ generated by algorithms FISTA\_exp and IFB\_AdapM\_exp, which slower than $R$-linear convergence for algorithms FB, FISTA\_R and IFB\_${\beta ^ * }$ from a theoretical point of view. However, we can see from Fig.\ref{fig:4} and Fig.\ref{fig:5} that IFB\_AdapM\_exp always has better performance than FISTA\_exp and sometimes, IFB\_AdapM\_exp performs better than other four algorithms and FISTA\_exp performs similar with IFB\_${\beta ^ * }$ but don't require the strong convex parameter.  Consequently, although the linear convergence rate is not reached, FISTA\_exp and IFB\_AdapM\_exp still have good numerical performances, and this adaptive modification scheme can significantly improve the convergence speed of IFB.

\begin{figure}[tbhp]
	\centering
	\subfloat
	{
		\includegraphics[width=6.0cm]{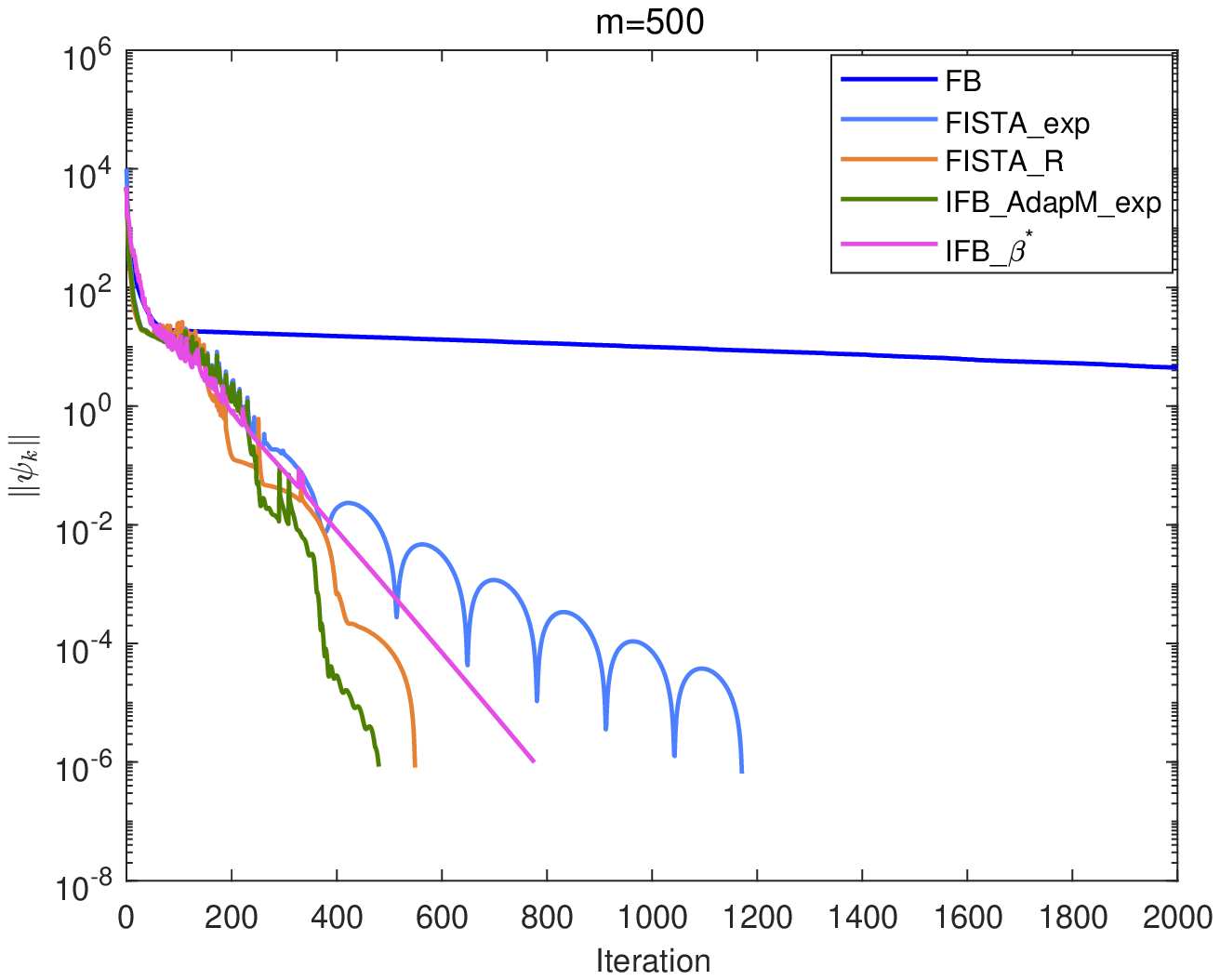}	
	}
	\subfloat
	{
		\includegraphics[width=6.0cm]{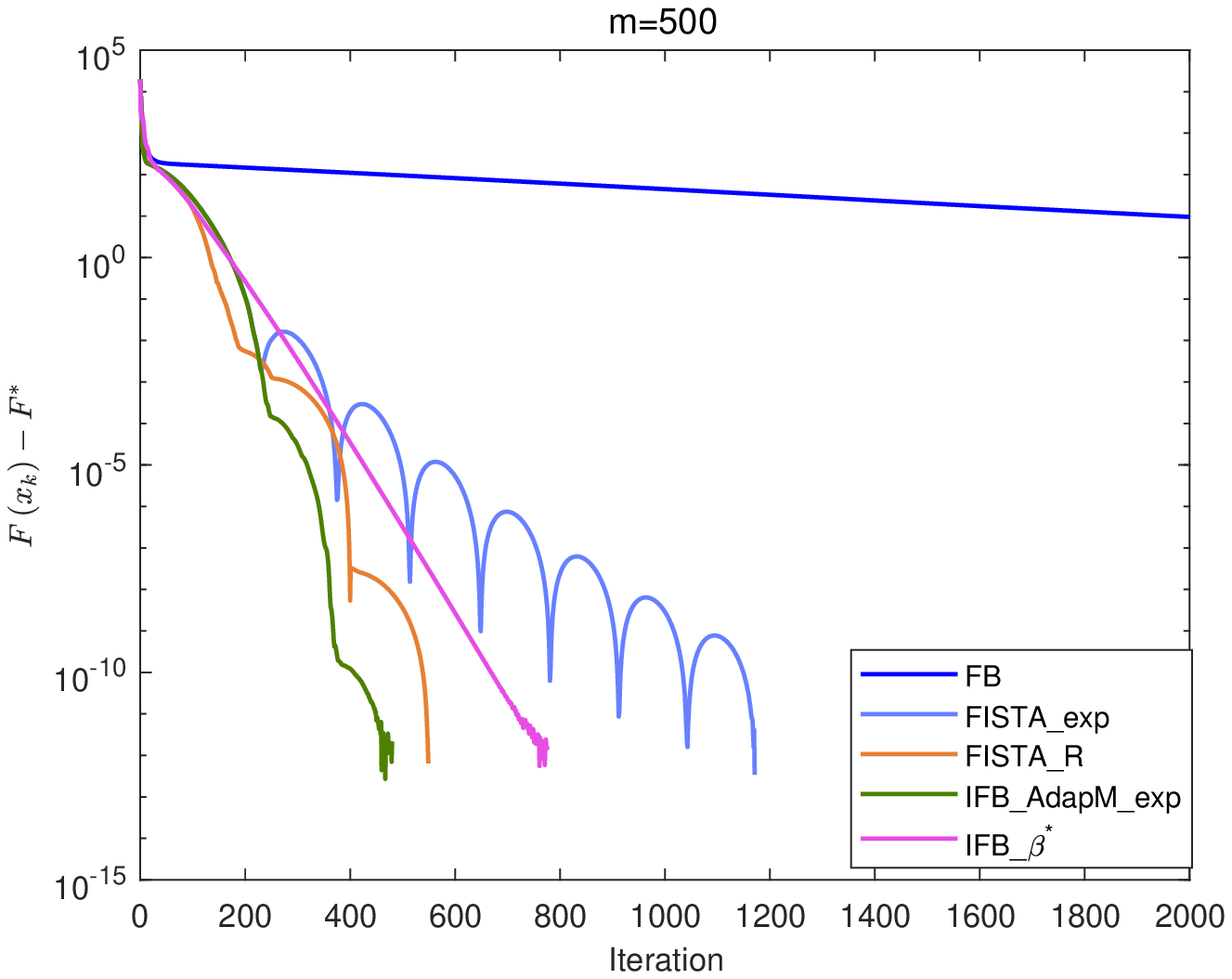}	
	}
	\caption{Computational results for the convergence of $\left\| {{\psi _k}} \right\| $ and $\left( {F\left( {{x_k}} \right) - F^*} \right).$}
	\label{fig:4}	
\end{figure}
\begin{figure}[tbhp]
	\centering
	\subfloat
	{
		\includegraphics[width=6.0cm]{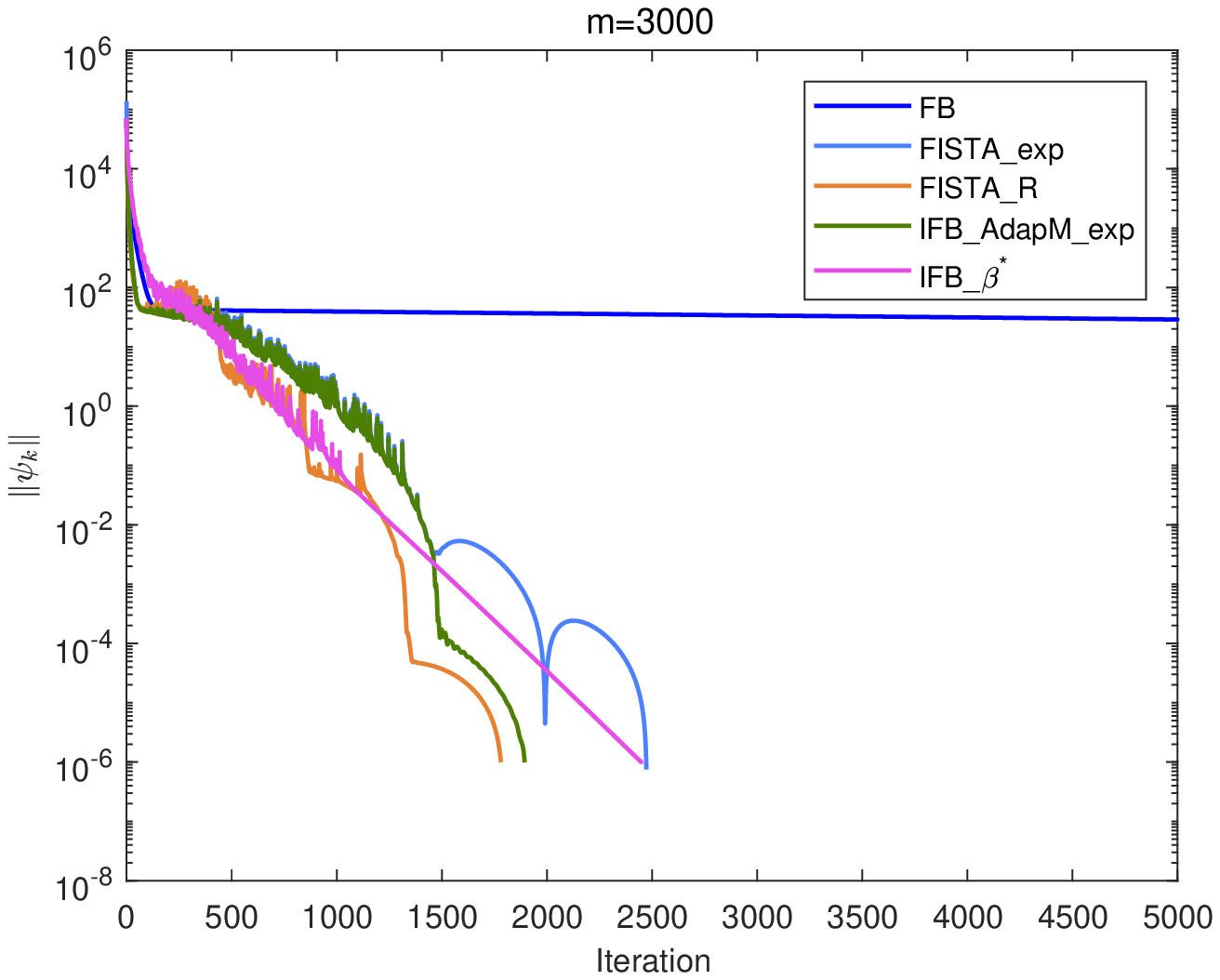}	
	}
	\subfloat
	{
		\includegraphics[width=6.0cm]{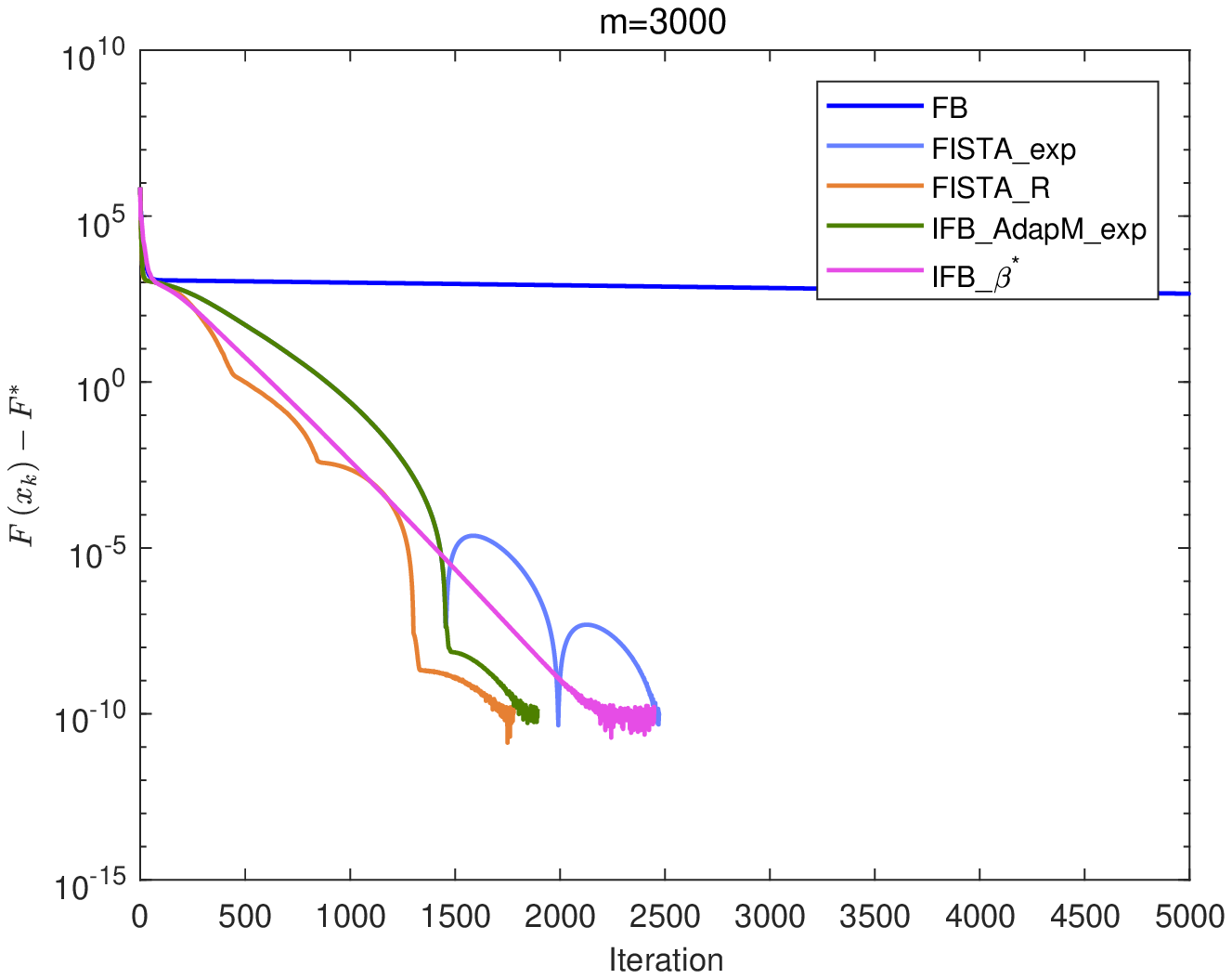}	
	}
	\caption{Computational results for the convergence of $\left\| {{\psi _k}} \right\| $ and $\left( {F\left( {{x_k}} \right) - F^*} \right).$}
	\label{fig:5}	
\end{figure}
\section{Conclusion}
In this paper, under the local error bound condition, we study the convergence results of IFBs with a class of abstract $t_k$ satisfying the assumption $A_2$ for solving the problem ($P$). We use a new method called ``comparison method'' to discuss the improved convergence rates of function values and sublinear rates of convergence of iterates generated by the IFBs with six choices of $t_k.$ In particular, we show that, under the local error bound condition, the strong convergence of iterates generated by the original FISTA can be established, the convergence rate of FISTA\_CD is actually related to the value of $a,$ and the sublinear convergence rates for both of function values and iterates generated by IFBs with $t_k$ in Case 1 and Case 3 can achieve $o\left( {\frac{1}{{{k^p}}}} \right)$ for any positive integer $p>1.$ 
Specifically, our results still hold for IFBs with an adaptive modification scheme.

\section*{Acknowledgements}

The work was supported by the National Natural Science Foundation of China (No.11901561), the Natural Science Foundation of Guangxi (No.2018GXNSFBA281180) and the Postdoctoral Fund Project of China (Grant No.2019M660833).

\clearpage

\appendix

\section{Proof of Lemma 2.4}
\label{A}
\begin{proof}	
	Assume by contradiction that $\mathop {\lim \inf }\limits_{k \to \infty } {s_k} = l <  + \infty .$ Notice that $l\ge 0$ since that $\left\{ {{s_k}} \right\}$ is a nonnegative sequence. Then, there exists a subsequence $\left\{ {{s_{{k_j}}}} \right\}$ such that \textcolor{red}{$\mathop {\lim }\limits_{j \to \infty } {s_{{k_j}}} = l.$}
	By the condition ${\alpha _k} = \frac{{{s_k} - 1}}{{{s_{k + 1}}}} \ge {\gamma _k},$ we have $\frac{{{s_{{k_j}}} - 1}}{{{s_{{k_j} + 1}}}} \ge {\gamma _{{k_j}}},$ then, combining with the fact that $\mathop {\lim }\limits_{k \to \infty } {\gamma _k} = 1$ from Remark 3, we deduce that 
	\textcolor{red}{\[\mathop {\lim \sup }\limits_{j \to \infty } {s_{{k_j} + 1}} \le \mathop {\lim }\limits_{j \to \infty } \frac{{{s_{{k_j}}} - 1}}{{{\gamma _{{k_j}}}}} = l - 1,\]} which leads to a contradiction that \textcolor{red}{$l = \mathop {\lim \inf }\limits_{k \to \infty } {s_k} \le \mathop {\lim \inf }\limits_{j \to \infty } {s_{{k_j} + 1}} \le \mathop {\lim \sup }\limits_{j \to \infty } {s_{{k_j} + 1}} \le l - 1.$} Hence, $\mathop {\lim }\limits_{k \to \infty } {s_k} =  + \infty .$
	
	Further, by the condition ${\alpha _k} = \frac{{{s_k} - 1}}{{{s_{k + 1}}}} \ge {\gamma _k},$ we get $\frac{{{s_{k + 1}}}}{{{s_k}}} \le \frac{1}{{{\gamma _k} + \frac{1}{{{s_{k + 1}}}}}}.$ Combining with $\mathop {\lim }\limits_{k \to \infty } {\gamma _k} = 1$ and $\mathop {\lim }\limits_{k \to \infty } {s_k} =  + \infty ,$ we obtain that $\mathop {\lim \sup }\limits_{k \to \infty } \frac{{{s_{k + 1}}}}{{{s_k}}} \le 1.$ Since that $\left\{ {{s_k}} \right\}$ is a nonnegative subsequence, we have $\mathop {\lim \sup }\limits_{k \to \infty } {\left( {\frac{{{s_{k + 1}}}}{{{s_k}}}} \right)^2} \le 1,$ which leads to the result that 
	\[\mathop {\lim \sup }\limits_{k \to \infty } \frac{{s_{k + 1}^2 - s_k^2}}{{s_k^2}} = \mathop {\lim \sup }\limits_{k \to \infty } {\left( {\frac{{{s_{k + 1}}}}{{{s_k}}}} \right)^2} - 1 \le 0.\]
\end{proof}

\end{document}